\documentclass[12pt,amssymb]{amsart}
\usepackage{amsfonts,latexsym,geometry,color}
\usepackage{amssymb,graphics,graphicx}
\usepackage{amsfonts,latexsym,geometry}
\usepackage{amssymb,graphics,graphicx}
\usepackage{amscd}
\usepackage{amssymb}
\usepackage{amsthm}
\usepackage[utf8]{inputenc} 
\usepackage[T1]{fontenc}
\usepackage{amsmath}
\usepackage{amsfonts}
\usepackage{amssymb}
\usepackage{graphicx} 
\usepackage{enumerate}
\newcommand\supp{\mathrm{supp}}
\newtheorem{theoreme}{Theorem}[section] %
\newtheorem{proposition}[theoreme]{Proposition} %
\newtheorem{corollary}[theoreme]{Corollary} %
\newtheorem{lemme}[theoreme]{Lemma} %
\newtheorem{definition}{Definition}[section] %
\newtheorem{remark}[theoreme]{Remark} %
\newtheorem{remarque}[theoreme]{Remark} %
\newcommand\mk{\medskip}
\newcommand\sk{\smallskip}
\newcommand\N {\mathbb{N}}
\newcommand\R{\mathbb{R}} 
 
\newcommand\dimm{\underline{\dim}_H}

\author{Edouard Daviaud}

\newcommand\ep{\varepsilon}
 
\RequirePackage{yhmath} 
\renewcommand\widering[1]{\ring{#1}}

\begin{document}
\title[Extraction of sequences of balls, and applications]{Extraction of  optimal subsequences of sequence of balls, and application to optimality estimates of  mass transference principles}
\maketitle
\begin{abstract}
In this article, we prove that from any sequence of balls whose
associated limsup set has full $\mu$-measure, one can extract a well-distributed
subsequence of balls. From this, we deduce the optimality of various lower bounds
for the Hausdorff dimension of limsup sets of balls obtained by mass transference
principles. One also gives a version of Borel-Cantelli lemma suitable for limsup sets of balls of full measure.
\end{abstract}
\section{Introduction}

Investigating Hausdorff dimensions of sets of points approximable at certain ``speed rate'' by a given sequence of points $(x_n)_{n\in\mathbb{N}}$  is an important topic  in Diophantine approximation (see \cite{BV} and \cite{BS2} among other references), in dynamical systems \cite{HV,LS,PR} and in multifractal analysis \cite{Ja,BS3,BS4}. These studies consists in general, knowing that $\mu(\limsup_{n\rightarrow+\infty}B_n:=B(x_n ,r_n))=1$ for a certain measure $\mu$ and a sequence of radius $(r_n)_{n\in\mathbb{N}}$,  in investigating the Hausdorff dimension of $\limsup_{n\rightarrow+\infty}U_n$ where  $U_n\subset B_n $. Typically  $U_n$ is a contracted ball inside $B_n$, but recently, general sets $U_n$ have been considered  \cite{KR,rect,ED1}.The so-called ubiquity theorems or mass transference principles mainly focus on finding a lower bound, using an adequate measure $\mu$, for the Hausdorff dimension (or Hausdorff measure) of those sets. However it is key in many situations to understand whether this lower bound is optimal or not. This article is dedicated to this problem.

 \sk

In order to do so, we introduce, given a probability measure $\mu$ on $\mathbb{R}^d$, the concept of $\mu$-asymptotically covering sequence of balls. This notion is a generalization of a covering property used in the $KGB$ Lemma stated in \cite{BV}.
For a sequence $(B_n)_{n\in\mathbb{N}}$, verifying this condition will be proved to be almost equivalent to verify $\mu(\limsup_{n \rightarrow+\infty}B_n)=1 $ (it is equivalent if the measure is doubling for instance, so that working under this settings is very reasonable).
 
As said above,  given a sequence of balls $(B_n)_{n\in\mathbb{N}}$ and another $(U_n \subset B_n)_{n\in\mathbb{N}}$, ubiquity theorems or mass transference principles give lower bounds for the Hausdorff dimension of $\limsup_{n\rightarrow+\infty}U_n$ when, roughly speaking, some information is known about the geometry of $ \limsup_{n\rightarrow+\infty}B_n.$ Of course, there is no reason in general for a lower-bound for $\dim_H (\limsup_{n\rightarrow+\infty}U_n)$ obtained only knowing that $\mu(\limsup_{n\rightarrow+\infty}B_n )=1$, to be sharp (i.e $=\dim_H (\limsup_{n\rightarrow+\infty}U_n)$). If one hopes such a lower-bound to be accurate,  the measure $\mu$ has to be particularly adapted, in some sense, to the pair of sequences  $((B_n)_{n\in\mathbb{N}},(U_n)_{n\in\mathbb{N}}).$ The approach adopted  in this article is to extract some sub-sequences from $(B_n)_{n\in\mathbb{N}}$ which are still $\mu$-a.c (or still verifies $\mu(\limsup_{n\rightarrow+\infty}B_n)=1$) but are adapted to the measure $\mu$. Applying mass transference principles (which are proved only for measures presenting enough self-similarity) to those sub-sequences, it will be proved that Hausdorff dimension of the limsup set associated with the corresponding $U_n$'s is given by the lower-bounds found in \cite{ED3} (so that it is also the case for the lower-bounds given in \cite{ED1, BS2, rect}).

This shows that those lower-bounds are sharp in a strong sens: for any self-similar measure $\mu$, any $\mu$-a.c sequence $(B_n)_{n\in\mathbb{N}}$, if one only considers the balls that are relevant for the measure $\mu$, the limsup set obtained by considering the sub-sequence of the  corresponding $U_n$'s ($(U_n)_{n\in\mathbb{N}}$ being also the sequence of sets involved in the articles mentioned)  has the expected dimension.

More precisely, it will be proven first that, under those very weak condition over a $\mu$-a.c  sequence of balls $(B_n)_{n\in \mathbb{N}}$, it is always possible to extract a sub-sequence $(B_{\phi(n)})_{n\in\mathbb{N}}$, still $\mu$-ac, "weakly redundant" (see Definition \ref{wr}) and such that the balls $(B_{\phi(n)})_{n\in\mathbb{N}}$ have prescribed behavior with respect to the measure $\mu$, roughly meaning (see Theorem \ref{theoremextra}) that the balls $(B_{\phi(n)})_{n\in\mathbb{N}}$ satisfies 
\begin{equation}
\label{roughcondi}
\vert B_{\phi(n)} \vert^{\overline{\dim}_P (\mu)} \lessapprox \mu(B_{\phi(n)})\lessapprox\vert B_{\phi(n)} \vert^{\underline{\dim}_H (\mu)}.
\end{equation}
In a second time, it will be proved that, for weakly redundant sequences  satisfying \eqref{roughcondi},  the Hausdorff dimension of $\limsup$ set associated with any sequence of shrunk balls or very thin rectangles $(R_n \subset B_n)_{n\in\mathbb{N}}$ (see Theorem \ref{rectssmajo}) can be bounded by above precisely by the lower-bound given in \cite{ED3}, which proves the optimality of those bounds.

\section{Notation and definition}
Let $d$ $\in\mathbb{N}$. \smallskip

For $n\in\mathbb{N}$, the set of dyadic cubes of generation $n$ of $\mathbb{R}^d$ is denoted $\mathcal{D}_n (\mathbb{R}^d)$ and defined as $\mathcal{D}_n (\mathbb{R}^d) = \left\{\prod_{i=1}^d [\frac{k_i}{2^{n}},\frac{k_{i}+1}{2^{n}})\right\}_{(k_i)_{i\in\mathbb{Z}}\in\mathbb{Z}^{\mathbb{Z}}}.$

For $x\in\mathbb{R}^{d}$, $r>0$,  $B(x,r)$ stands for the closed ball of ($\mathbb{R}^{d}$,$\vert\vert \ \ \vert\vert_{\infty}$) of center $x$ and radius $r$.
 Given a ball $B$, $\vert B\vert$ is the diameter of $B$.
 
For $t\geq 0$, $\delta\in\mathbb{R}$ and $B:=B(x,r)$,   $t B$ stand for $B(x,t r)$, i.e. the ball with same center as $B$ and radius multiplied by $t$,   and the  $\delta$-contracted  $B^{\delta}$ is  defined by $B^{\delta}=B(x ,r^{\delta})$.
\smallskip

Given a set $E\subset \mathbb{R}^d$, $\overset{\circ}{E}$ stands for the  interior of the $E$, $\overline{E}$ its  closure and $\partial E$ is the boundary of $E$, i.e, $\partial E =\overline{E}\setminus \overset{\circ}{E}.$
\smallskip

The $\sigma$-algebra of  Borel sets of $\mathbb{R}^d$ is denoted by  $\mathcal{B}(\mathbb{R}^d)$,
$\mathcal{L}^d$ is the Lebesgue measure on $\mathcal{B}(\mathbb{R}^d)$ and
$\mathcal{M}(\mathbb{R}^d)$ stands for the set of Borel probability measure over $\mathbb{R}^d.$
\smallskip

For $\mu \in\mathcal{M}(\mathbb{R}^d)$,   $\supp(\mu):=\left\{x: \ \forall r>0, \ \mu(B(x,r))>0\right\}$ is the topological support of $\mu.$
\smallskip

 Given $E\in\mathcal{B}(\mathbb{R}^d)$, $\dim_{H}(E)$ and $\dim_{P}(E)$ denote respectively  the Hausdorff   and the packing dimension of $E$.
\smallskip

\subsection{Definition and recalls}
\begin{definition}
\label{hausgau}
Let $\zeta :\mathbb{R}^{+}\mapsto\mathbb{R}^+$ be an increasing mapping  verifying $\zeta (0)=0$. The Hausdorff measure at scale $t\in(0,+\infty)$ associated with $\zeta$ of a set $E$ is defined by 
\begin{equation}
\label{gaug}
\mathcal{H}^{\zeta}_t (E)=\inf \left\{\sum_{n\in\mathbb{N}}\zeta (\vert B_n\vert) : \ (B_n)_{n\in \N}  \mbox{ closed  balls, }\vert B_n \vert \leq t \text{ and } E\subset \bigcup_{n\in \mathbb{N}}B_n\right\}.
\end{equation}
The Hausdorff measure associated with $\zeta$ of a set $E$ is defined by 
\begin{equation}
\mathcal{H}^{\zeta} (E)=\lim_{t\to 0^+}\mathcal{H}^{\zeta}_t (E).
\end{equation}
\end{definition}

For $t\in (0,+\infty)$, $s\geq 0$ and $\zeta:x\to x^s$, one simply uses the usual notation $\mathcal{H}^{\zeta}_t (E)=\mathcal{H}^{s}_t (E)$ and $\mathcal{H}^{\zeta} (E)=\mathcal{H}^{s} (E).$ In particular, the $s$-dimensional Hausdorff outer measure at scale $t\in(0,+\infty]$  of the set $E$ is defined by
\begin{equation}
\label{hcont}
\mathcal{H}^{s}_{t}(E)=\inf \left\{\sum_{n\in\mathbb{N}}\vert B_n\vert^s : \ (B_n)_{n\in \N}  \mbox{ closed  balls, }\vert B_n \vert \leq t \text{ and } E\subset \bigcup_{n\in \mathbb{N}}B_n\right\}. 
\end{equation}

\begin{definition} 
\label{dim}
Let $\mu\in\mathcal{M}(\mathbb{R}^d)$.  
For $x\in \supp(\mu)$, the lower and upper  local dimensions of $\mu$ at $x$ are   
\begin{align*}
&\underline\dim (\mu,x)=\liminf_{r\rightarrow 0^{+}}\frac{\log\mu(B(x,r))}{\log r} \\
 \mbox{ and } \ \ \ \  & \overline\dim (\mu,x)=\limsup_{r\rightarrow 0^{+}}\frac{\log \mu(B(x,r))}{\log r}.
 \end{align*}
Then, the lower and upper dimensions of $\mu$  are defined by 
\begin{equation}
\label{dimmu}
\dimm(\mu)=\mbox{infess}_{\mu}(\underline\dim (\mu,x))  \ \ \mbox{ and } \ \ \overline{\dim}_P (\mu)=\mbox{supess}_{\mu}(\overline\dim (\mu,x)).
\end{equation}
\end{definition}

It is known that  (for more details see \cite{F})
$$\dimm(\mu)=\inf_{E\in\mathcal{B}(\mathbb{R}^d):\ \mu(E)>0}\dim_{H}(E)  \ \ \ \mbox{ and } \ \ \overline{\dim}_P (\mu)=\inf_{E\in\mathcal{B}(\mathbb{R}^d):\ \mu(E)=1}\dim_{P}(E).$$ 

A measure verifying $\dimm(\mu)=\overline{\dim}_P (\mu):=\alpha$ will be called an $\alpha$ exact dimensional measure. From Definition \ref{dim}, such measures verify, for $\mu$-almost every $x\in \mathbb{R}^d$, $\lim_{r\rightarrow 0^{+}}\frac{\log \mu(B(x,r))}{\log r}=\alpha.$

\subsection{Main statements}

Before stating the Theorems proved in this article, one starts by recalling the following definition, introduced in \cite{BS}. 

\begin{definition}
\label{wr}
Let $\mathcal{B}=( B _n =:B(x_n ,r_n))_{n\in\mathbb{N}}$ be a family of balls in $\mathbb{R}^d$. Denote by $\mathcal{T}_k(\mathcal{B})=\left\{B_n \ :  \   2^{-k-1}< r_{n}\leq 2^{-k}\right\}.$ 
The family $\mathcal{B}$ is said to be weakly redundant when for all $k$, there exists an integer $J_k$ and $\mathcal{T}_{k,1}(\mathcal{B}),..,\mathcal{T}_{k,J_k}(\mathcal{B})$ a partition of $\mathcal{T}_k(\mathcal{B})$ such that:\medskip
\begin{itemize}
 \item[\textbf{$(C_1)$}] $\mathcal{T}_k(\mathcal{B})=\bigcup_{1\leq j\leq J_k} T_{k,j}(\mathcal{B}),$\medskip
  \item[\textbf{$(C_2)$}] For every $1\leq j\leq J_k$ and every pair of balls $B\neq B^{\prime}\in \mathcal{T}_{k,j}(\mathcal{B})$, 
  $B\cap B^{\prime}=\emptyset,$\medskip
  \item[\textbf{$(C_3 )$}] $\lim_{k\rightarrow +\infty}\frac{\log_{2}(J_k)}{k}=0.$
  \end{itemize}
  \end{definition}

So, a sequence of balls $(B_n)_{n\in\mathbb{N}}$ is  weakly redundant  when   at each scale  $2^{-k}$, the balls of the family $\left\{B_n \right\}_{n\in\mathbb{N}}$ that have radii $\approx 2^{-k}$ can be sorted in  a relatively small number of families of pairwise disjoint balls.

The main  property we introduce  for a sequence of balls $\mathcal{B}=(B_n)_{n\in\mathbb{N}}$  is meant to ensure  that any set can be covered   efficiently  by the limsup of the $B_n$'s, with respect to a measure $\mu$.  This property  is  a general version of the KGB Lemma of  Beresnevitch and Velani, stated in \cite{BV}, using a Borel probability measure $\mu$. Such properties (like the KGB Lemma) are   usually   key   (cf \cite{Ja,BV,BS} for instance) to prove ubiquity or mass transference results.
\begin{definition} 
\label{ac}
Let   $\mu\in \mathcal{M}(\mathbb{R}^d)$. The  sequence $\mathcal{B}= (B_n)_{n\in\mathbb{N}}$ of balls of $\mathbb{R}^d$   is said to be $\mu$-asymptotically covering (in short,  $\mu$-a.c) when  there exists  a constant $C>0$ such that for every open set $\Omega\subset \mathbb{R}^d $ and $g\in\mathbb{N}$, there is an integer  $N_\Omega \in\mathbb{N}$ as well  as $g\leq n_1 \leq ...\leq n_{N_\Omega}$ such that: 
\begin{itemize}
\item $\forall \, 1\leq i\leq N_\Omega$, $B_{n_i}\subset \Omega,$\medskip
\item $\forall \, 1\leq i\neq j\leq N_\Omega$, $B_{n_i}\cap B_{n_j}=\emptyset,$\medskip
\item  one has
\begin{equation}
\label{majac}
\mu\left(\bigcup_{1\leq i\leq N_\Omega}B_{n_i} \right)\geq C\mu(\Omega).
\end{equation}
\end{itemize}
\end{definition}

In other words, for any open set $\Omega$ and any $g>0$, there exists a finite set of disjoint balls of $\left\{B_n\right\}_{n\geq g}$ covering a large part of $\Omega$ from the $\mu$-standpoint.

This notion of $\mu$-asymptotically covering  is related to the way the balls of $\mathcal{B}$ are distributed according to the measure $\mu$. In particular, given a measure $\mu$, this property  is  slightly stronger  than being of $\limsup$ of full $\mu$-measure, as illustrated by the following Theorem.

\begin{theoreme} 
\label{equiac}
Let   $\mu\in\mathcal{M}(\mathbb{R}^d)$  and $\mathcal{B} =(B_n :=B(x_n ,r_n))_{n\in\mathbb{N}}$ be a sequence of balls of  $\mathbb{R}^d$ with $\lim_{n\to +\infty} r_{n}= 0$.

\begin{enumerate}
\smallskip
\item
If $\mathcal{B} $ is $\mu$-a.c, then $\mu(\limsup_{n\rightarrow+\infty}B_n)=1.$
\smallskip
\item
 If there exists $v<1$ such that $ \mu \big(\limsup_{n\rightarrow+\infty}(v B_n) \big)=1$, then $\mathcal{B} $ is $\mu$-a.c. 
\end{enumerate}

\end{theoreme}
Moreover, it results from the proof of the KGB-Lemma \cite{BV} that if the $\mu$ is doubling, $\mu\Big(\limsup_{n\rightarrow+\infty}B_n \Big)=1 \Leftrightarrow (B_n)_{n \in\mathbb{N}}$ is $\mu$-a.c.

\sk 

An interesting consequences of Theorem \ref{equiac} is the following version of Borel-Cantelli lemma which is extends the case where $\mu$ is doubling, established \cite{BVBC}.
\begin{proposition}
Let $(B_n)_{n\in\mathbb{N}}$ be a sequence of closed balls satisfying $\vert B_n \vert\to 0$ and $ \mu\in\mathcal{M}(\mathbb{R}^d)$ be a probability measure. 
\begin{itemize}
\item[(A):] Assume that $(B_n)_{n\in\mathbb{N}}$ is $\mu$-a.c, then, there exists $C>1$ such that for any open ball $B$, there exists a sub-sequence of $(B_n)_{n\in\mathbb{N}}$,  $(L_{n,B})_{n\in\mathbb{N}}$ satisfying, for any $n\in\mathbb{N}$, $L_{B,n} \subset B$ and 
\begin{equation}
\sum_{n \geq 0}\mu(L_{B,n})=+\infty 
\end{equation}
and for infinitely many Q,
\begin{equation}
\sum_{s,t=1}^{Q}\mu(L_{B,s}\cap L_{B,t})\leq \frac{C}{\mu(B)}\left(\sum_{n=1}^{Q}\mu(L_{B,n})\right)^2.
\end{equation}
\item[(B):] Assume that there exists $C>1$ such that for any open ball $B$, there exists a sub-sequence of $(B_n)_{n\in\mathbb{N}}$ $(L_{n,B})_{n\in\mathbb{N}}$ with, for any $n\in\mathbb{N},$ $L_{n,B}\subset B$, satisfying  \eqref{eq1} and \eqref{eq2}, then $\mu(\limsup_{n\rightarrow+\infty}B_n)=1,$ so that, for any $\kappa >1,$ $(\kappa B_n)_{n\in\mathbb{N}}$ is $\mu$-a.c.
\end{itemize} 
\end{proposition}

\sk

One now states the main result about extraction of sub-sequences of balls of this article.

\begin{theoreme}
\label{theoremextra}
Let $\mu \in\mathcal{M}(\mathbb{R}^d)$ Let $(B_n)_{n\in\mathbb{N}}$ be a sequence of balls of $\mathbb{R}^d .$
\begin{enumerate}
\item If $(B_n)_{n\in\mathbb{N}}$ is $\mu$-a.c, then there exists a $\mu$-a.c sub-sequence $(B_{\phi(n)})_{n\in\mathbb{N}}$ which is weakly redundant.\medskip
\item If there exists $v<1$ such that $\mu(\limsup_{n\rightarrow+\infty}vB_n)=1$, then there exists a $\mu$-a.c sub-sequence $(B_{\phi(n)})_{n\in\mathbb{N}}$ verifying
\begin{equation}
\label{condimesu}
  \underline{\dim}_H (\mu)\leq \liminf_{n\rightarrow+\infty}\frac{\log \mu(B_{\phi(n)})}{\log\vert B_{\phi(n)}\vert}\leq \limsup_{n\rightarrow+\infty}\frac{\log \mu(B_{\phi(n)})}{\log \vert B_{\phi(n)}\vert}\leq \overline{\dim}_P (\mu).
  \end{equation}
\end{enumerate}
\end{theoreme}
\begin{remarque}

  Theorem \ref{theoremextra} implies in particular that if the sequence of balls  $(B_n)_{n \in\mathbb{N}}$ verifies $\mu( \limsup_{n\rightarrow+\infty}vB_n)=1$, for some $v<1$,  it is possible to extract a $\mu$-a.c sub-sequence verifying items $(1)$ and $(2)$. 

\end{remarque}

\bigskip

Section \ref{secequiac} and Section \ref{sec-extrac} are respectively dedicated to the proof of Theorem \ref{equiac} and Theorem \ref{theoremextra}. 

Section \ref{sec-exam} provides some explicit applications of Theorem \ref{theoremextra}. 

In the last section, Section \ref{sec-upper}, one proves Theorem \ref{majoss} and some final remark are given about Corollary \ref{equass1}.  

\subsection{Application to the study of the optimality of lower-bounds in obtained via mass transference principles}
\subsubsection{An upper-bound for ubiquity Theorem in the self-similar case}

In this section, one shows how the previous extraction theorem can be used to investigate optimal bounds in inhomogeneous mass transference principles. Let us recall first the definition of the following geometric quantity, introduced in \cite{ED3}.
\begin{definition} 
\label{mucont}
{Let $\mu \in\mathcal{M}(\R^d)$, and $s\geq 0$.
The $s$-dimensional $\mu$-essential Hausdorff content at scale $t\in(0,+\infty]$ of a set $A\subset \mathcal B(\R^d)$ is defined as}
{\begin{equation}
\label{eqmucont}
 \mathcal{H}^{\mu,s}_{t}(A)=\inf\left\{\mathcal{H}^{s}_{t}(E): \ E\subset  A , \ \mu(E)=\mu(A)\right\}.
 \end{equation}}
\end{definition}

Let us also recall the definition of a self-similar measure.

%

\begin{definition}
\label{def-ssmu}  A self-similar IFS is a family $S=\left\{f_i\right\}_{i=1}^m$ of $m\ge 2$ contracting similarities  of $\mathbb{R}^d$. 

Let $(p_i)_{i=1,...,m}\in (0,1)^m$ be a positive probability vector, i.e. $p_1+\cdots +p_m=1$.

The self-similar measure $\mu$ associated with $ \left\{f_i\right\}_{i=1}^m$  and $(p_i)_{i=1}^m$ is the unique probability measure such that 
\begin{equation}
\label{equass}
\mu=\sum_{i=1}^m p_i \mu \circ f_i^{-1}.
\end{equation}

The topological support of $\mu$ is the attractor of $S$, that is the unique non-empty compact set $K\subset X$ such that  $K=\bigcup_{i=1}^m f_i(K)$.

\end{definition}

The existence  and uniqueness of $K$ and $\mu$ are standard results \cite{Hutchinson}. Recall that due to a result by Feng and Hu \cite{FH} any self-similar measure is exact dimensional. 

Our goal is to investigate whether the lower-bound given by the following theorem, proved in \cite{ED3}, is sharp.

\begin{theoreme}
\label{zzani}
Let $\mu\in\mathcal{M}(\R^d)$ be a self-similar measure and $\mathcal{B}= (B_n)_{n\in\mathbb{N}}$ be a   $\mu$-a.c. sequence of  closed balls of~$\R^d$ centered in $\supp (\mu)$.
Let $\mathcal{U}=(U_n)_{n\in\mathbb{N}}$ be a sequence of open sets such that 
$U_n \subset B_n$ for all $n\in \mathbb N$, and $0\leq s\leq \dim(\mu)$. If, for every  $n\in\mathbb{N}$ large enough, $\mathcal{H}^{\mu,s}_{\infty}(U_n)\geq \mu(B_n) $, then 
$$\dim_{H}(\limsup_{n\rightarrow +\infty}U_n)\geq s .
$$

\end{theoreme}

One now states the main result of this section.

\begin{theoreme}
\label{majoss}
Let $\mu \in \mathcal{M}(\R^d)$ be a self-similar measure, $K$ its support and $(B_n)_{n\rightarrow+\infty}$ be a weakly redundant sequence of balls of $\R^d$ verifying $\vert B_n \vert \to 0.$ Let $(U_n)_{n\in \mathbb{N}}$ be a sequence of open sets satisfying $U_n \subset B_n$. For any  $0\leq s <\dim (\mu)$ such that, for all large enough  $n\in\mathbb{N}$, $\mathcal{H}^{\mu,s}_{\infty}(U_n)\leq \mu(B_n)$, 
\begin{equation}\label{uppernound18}
\dim_H (\limsup_{n\rightarrow+\infty}U_n \cap K)\leq s.
\end{equation}

\end{theoreme}
Note that if, for any $n\in\mathbb{N}$, the ball $B_n$ intersects $K$, $\limsup_{n\rightarrow+\infty}U_n \cap K=\limsup_{n\rightarrow+\infty}U_n$. 

Theorem \ref{zzani} and Theorem \ref{majoss} yields the following useful corollary.
\begin{corollary}
\label{equass1}
Let $\mu \in\mathcal{M}(\mathbb{R}^d)$ be a self-similar measure. Let $(B_n)_{n\in\mathbb{N}}$ be a weakly redundant $\mu$-a.c sequence of balls satisfying $\vert B_n \vert \to 0$ and $B_n \cap K \neq \emptyset$ for any $n\in\mathbb{N}.$ 

Let $(U_n)_{n\in\mathbb{N}}$ be a sequence of open sets satisfying that, for any $n\in\mathbb{N}$, $U_n \subset B_n$. 

Assume that there exists $s_0$ such that
\begin{itemize}
\item[•] for any $s<s_0$, for $n$ large enough, $\mathcal{H}^{\mu,s}_{\infty}(U_n)\geq \mu(B_n),$\mk
\item[•]  for any $s>s_0$, for $n$ large enough, $\mathcal{H}^{\mu,s}_{\infty}(U_n)\leq \mu(B_n).$\mk 
\end{itemize}

Then by Theorem \ref{zzani} and Theorem \ref{majoss}, $$\dim_H (\limsup_{n\rightarrow+\infty}U_n)=s_0 .$$
\end{corollary}

\begin{remark}
\label{remajo}
It is easily seen from the proof that the condition $\mathcal{H}^{\mu,s}_{\infty}(U_n)\leq \mu(B_n)$ in Theorem \ref{majoss} can be weakened into $\liminf_{n\rightarrow+\infty}\frac{\log \mathcal{H}^{\mu,s}_{\infty}(U_n)}{\log \mu(B_n)}\geq 1$.

\end{remark}

\subsubsection{Application in the case of balls and rectangles}

We can now show in which sense, in view of Theorem~\ref{theoremextra},   Theorem~\ref{zzani} is sharp by applying Corollary \ref{equass1} to the specific cases where the sets $U_n$ are balls or rectangles.    

\begin{corollary}\label{sharp}
Let $\mu \in\mathcal{M}(\R^d)$  be a self-similar measure of support $K$ and $\mathcal{B}=(B_n)_{n\in\mathbb{N}}$ be a  sequence of balls centered in $K$  satisfying $\vert B_n\vert \to 0 $ and $\mu\Big( \limsup_{n\rightarrow+\infty}B_n\Big)=1.$ Then \cite{ED3},
$$\dim_H (\limsup_{n\rightarrow+\infty}B_n ^{\delta})\geq \frac{\dim(\mu)}{\delta}.
$$
Assume furthermore that $\mathcal{B} $ is weakly redundant and  $\limsup_{n\rightarrow+\infty}\frac{\log \mu(B_n )}{\log(\vert B_n \vert)}=\dim (\mu)$,  then for every $\delta\geq 1$, 
$$\dim_H (\limsup_{n\rightarrow+\infty}B_n ^{\delta})=\frac{\dim(\mu)}{\delta}.
$$
\end{corollary}

\begin{corollary}
 \label{rectssmajo}
 Let $\mu$ be a self-similar measure verifying that its support, $K$, is the closure of its interior. Let $1\leq \tau_1\leq...\leq \tau_d$, $\tau=(\tau_1,..., \tau_d)$ and $(B_n :=B(x_n,r_n))_{n\in\mathbb{N}}$ be a sequence of balls of $\R^d$ satisfying $r_n \to 0$, $\mu(\limsup_{n\rightarrow+\infty}B_n)=1.$  Define $R_n =\widering{R}_{\tau}(x_n,r_n),$ where $
R_{\tau}(x_n,r_n)=x_n+\prod_{i=1}^d [-\frac{1}{2}r_n^{\tau_i},\frac{1}{2}r_n^{\tau_i}].$
  Then \cite{ED3}
 \begin{equation}
 \dim_H (\limsup_{n\rightarrow+\infty}R_n)\geq \min_{1\leq i\leq d }\left\{\frac{\dim (\mu)+\sum_{1\leq j\leq i}\tau_i-\tau_j}{\tau_i}\right\}.
 \end{equation}
 Assume furthermore that $(B_n)_{n\in\mathbb{N}}$ is weakly redundant and $\lim_{n\rightarrow+\infty}\frac{\log \mu(B_n)}{\log \vert B_n \vert}=\dim (\mu)$, then 
 \begin{equation}
 \dim_H (\limsup_{n\rightarrow+\infty}R_n)= \min_{1\leq i\leq d }\left\{\frac{\dim (\mu)+\sum_{1\leq j\leq i}\tau_i-\tau_j}{\tau_i}\right\}.
 \end{equation}
 \end{corollary}
%
\begin{remarque}

\begin{itemize}
\item[•] Corollaries \ref{sharp} and \ref{rectssmajo} are direct consequences of second item of Remark 5.1 and Remark 5.3 in \cite{ED3}, together with Corollary \ref{equass1} (applied to, respectively, $s_0 =\frac{\dim (\mu)}{\delta}$ and $s_0 =s(\mu,\tau)$).\mk
\item[•] Note that, by Theorem \ref{theoremextra} combined with Corollary \ref{sharp} and Corollary \ref{rectssmajo}, for any sequence of balls $(B_n)_{n\in\mathbb{N}}$ satisfying $\mu\Big(\limsup_{n\rightarrow+\infty}\frac{1}{2} B_n\Big)=1$  ($\mu$ a self-similar measure satisfying the of hypothesis of Corollaries \ref{sharp} or \ref{rectssmajo} for $\mu$), it is always possible to extract a $\mu$-a.c sub-sequence of balls so that the Hausdorff dimension of the limsup set associated with corresponding $U_n$'s is the bound stated in \cite{ED3} and recalled in those corollaries. This in particular proves that those bounds are sharp.  \medskip
\item[•]In the case of the Lebesgue measure, it is always verified that $\lim_{n\rightarrow+\infty}\frac{\log \mu(B_n)}{\log \vert B_n \vert}=\dim (\mu).$ As a consequence, the lower-bound provided by Theorem \ref{zzani} (which is established in \cite{KR}) in the case of balls or rectangles is precisely the dimension of $\limsup_{n\rightarrow+\infty}U_n$ as soon as the sequence $(B_n)$ is weakly redundant. More explicitly, given a weakly redundant sequence of balls $(B_n)_{n\in\mathbb{N}}$ of $[0,1]^d$ satisfying $\vert B_n \vert \to 0$ and $ \mathcal{L}^d (\limsup_{n\rightarrow+\infty}B_n)=1$, for any sequence rectangles associated with a vector $\tau$ as in Theorem \ref{rectssmajo}, one has

$$\dim_H ( \limsup_{n\rightarrow+\infty}R_n)= \min_{1\leq i\leq d }\left\{\frac{d+\sum_{1\leq j\leq i}\tau_i-\tau_j}{\tau_i}\right\}.$$
\end{itemize}
\end{remarque}
\medskip

Section \ref{secpequiac} is dedicated to the proof of Theorem \ref{equiac}. In the next section, Section  \ref{sec-extrac}, Theorem \ref{theoremextra} is established. Then some explicit examples of application of Theorem \ref{theoremextra} are given in Section \ref{sec-exam}. 

In the penultimate section, Section \ref{sec-upper}, Theorem \ref{majoss} is proved. 

The last section, Section \ref{conclu}, draws some conclusions and gives some perspectives about the results established in this article.

\section{Proof of Theorem \ref{equiac}}
\label{secpequiac}
\subsection{A useful modified version of Besicovitch covering Lemma}
\label{secbesi}
One focuses on a  modified version of Besicovitch's covering Lemma.

\begin{proposition} 
\label{besimodi}
For any $0<v\leq 1$, there exists $Q_{d,v} \in\mathbb{N}^{\star}$, a constant depending only on the dimension $d$ and $v$, such that for every $E\subset [0,1]^{d}$,  for every set $\mathcal{F}=\left\{B(x, r_{(x)} ): x\in E,  r_{(x)} >0 \right\}$, there exists $\mathcal{F}_1,...,\mathcal{F}_{Q_{d,v}}$ finite or countable sub-families of $\mathcal{F}$ such that:\medskip
\begin{itemize}

\item
$\forall 1\leq i\leq Q_{d,v}$, $\forall L\neq L'\in\mathcal{F}_i$, one has $\frac{1}{v}L \cap \frac{1}{v}L'=\emptyset.$\medskip

\item 
$E$ is covered by the families $\mathcal{F}_i$, i.e.
\begin{equation}\label{besi}
 E\subset  \bigcup_{1\leq i\leq Q_{d,v}}\bigcup_{L\in \mathcal{F}_i}L.
 \end{equation}
\end{itemize}
\end{proposition}
 
The case $v=1$ corresponds to  the standard Besicovich's covering lemma (see \cite{Ma}, Chapter 2, pp. 28-34 for instance).

A first step toward   Proposition \ref{besimodi} is the  next lemma, that allows to split  a given family of "weakly" overlapping balls into a finite number of families of disjoint balls. 
\begin{lemme} 
\label{fami}
Let $0<v<1$ and $\mathcal{B} =(B_n)_{n\in\mathbb{N}}$   be a countable family of balls such that $\lim_{n\to +\infty} |B_n|=0$, and  for every  $n\neq n' \in \N$,   $vB_n \cap vB_n' =\emptyset$.

   There exists $\gamma_{d,v}+1$ ($\gamma_{d,v}$ being the constant appearing in Lemma \ref{dimconst} below) sub-families of $\mathcal{B}$, $(\mathcal{F}_i)_{1\leq i\leq \gamma_{d,v}+1}$, such that: 
\begin{itemize}
\item 
$\mathcal{B}=\bigcup_{1\leq i\leq \gamma_{d,v}+1}\mathcal{F}_i$,
\item
$\forall  \, 1\leq i\leq \gamma_{d,v}+1$, $\forall L\cap L' \in\mathcal{F}_i$, one has $L \cap L'=\emptyset.$
\end{itemize} 
\end{lemme} 

\begin{proof} 
The proof is based on the following lemma,   whose proof can be found in \cite{Ma}, Lemma 2.7, pp.30 - there,   the result is obtained for   $v= {1}/{2}$ but the proof remains valid for any $v<1$.
\begin{lemme} \label{dimconst}
For any $0<v\leq 1$ there exists a constant $\gamma_{v,d} >0 $ depending only on $v$ and the dimension $d$ only, satisfying the following: if  a  family of balls $\mathcal{B} =(B_n )_{n\in\mathbb{N}}$  and a ball $B$  are such that 

\begin{itemize}
\item
$\forall \ n\geq 1$, $\vert B_n\vert \geq \frac{1}{2}\vert B\vert,$

\item
 $\forall \ n_1 \neq n_2 \geq 1 $, $vB_{n_1}\cap vB_{n_2}=\emptyset,$ 
\end{itemize}
then $B$  intersects  at most $\gamma_{v,d}$ balls of $\mathcal{B}$.  
\end{lemme}

The families $\mathcal{F}_1 ,..., \mathcal{F}_{\gamma_{d,v}+1}$ are built recursively.

For $k\in\mathbb{N}$, call  $\mathcal{G}^{(k)}=\left\{L\in\mathcal{F}: 2^{-k-1} <  \vert L\vert \leq 2^{-k}\right\}$. Notice  that, because $\lim_{n\to +\infty} |B_n|=0$, each $\mathcal{G}^{(k)}$ is empty or finite.

Observe  first that for every $k\in\mathbb{N}$ and every ball $B\in\mathcal{G}^{(k)}$,  and every pair of balls $B_1 \neq B_2\in\bigcup_{k'\leq k}\mathcal{G}^{(k')}\setminus\left\{B\right\}$, one has $vB_1 \cap v B_2=\emptyset$ and for $i=1,2$, $\vert B_i \vert \geq \frac{\vert B \vert }{2}$. By Lemma \ref{dimconst}, this implies  that $B$ intersects at most $\gamma_{d,v}$ balls of $\bigcup_{k'\leq k}\mathcal{G}^{(k')}\setminus\left\{B\right\}$. 

To get Lemma \ref{fami}, we are going to sort the balls of $\bigcup_{k '\leq k}\mathcal{G}^{(k')}$ recursively on $k$ into  families $\mathcal{F}_1 ,..., \mathcal{F}_{\gamma_{d,v}+1}$ of   pairwise disjoint  balls. At each step, a new ball $B$ will be added to one of those families of balls $\mathcal{F}_{i}$ and the resulting family, $\mathcal{F}_i \bigcup \left\{B\right\}$ will be denoted again by $\mathcal{F}_i$.

\mk

Let $k_0$ be the smallest integer such that $\mathcal{G}^{(k_0)}$ is non-empty. Consider an arbitrary $L_0  \in\mathcal{G}^{(k_0)}$. By Lemma \ref{dimconst}, $L_0$ intersects  $n_{0}\leq\gamma_{d,v}$ other balls of $\mathcal{G}^{(k_0)}$, that are denoted by  $L_1 ,...,L_{n_0}$.  The sets  $\mathcal{F}_i$ are then set as follows:  
\begin{itemize}
\item
$\forall \, 1\leq i\leq n_0$, $\mathcal{F}_i =\left\{L_i \right\}$,
\item
$\forall \  n_0 +1\leq i\leq \gamma_{d,v}$, $\mathcal{F}_i=\emptyset$,
\item
 $\mathcal{F}_{\gamma_{d,v}+1}=\left\{L_0\right\}.$ 
\end{itemize}

Further, consider $\widetilde L \notin   \bigcup_{0\leq i\leq n_0}\left\{L_i\right\}$ (whenever such an $\widetilde L$ exists). The same argument (Lemma \ref{dimconst}) ensures that   $\widetilde L $ intersects at most $\gamma_{d,v}$ balls of $\mathcal{G}^{(k_0)}$.

In particular there must exists $1\leq i\leq \gamma_{d,v}+1$ such that for every $L\in\mathcal{F}_i$, $\widetilde L \cap L=\emptyset$. Choosing arbitrarily one of those indices $i$, one adds $\widetilde L$ to   $\mathcal{F}_i :=\widetilde L\bigcup \mathcal{F}_i $ (we keep the same name for this new family).

\mk

The same argument remains valid for any other ball $L^{\prime \prime} \notin \bigcup_{1\leq j\leq \gamma_{d,v}+1}\bigcup_{L\in\mathcal{F}_j}\left\{L\right\}$. Hence, proceeding recursively on all  balls of $\mathcal{G}^{(k_0)}$ allows to sort the balls of    $\mathcal{G}^{(k_0)}$ into $ \gamma_{d,v}+1$ families $(\mathcal{F}_i)_{1 \leq i \leq \gamma_{d,v}+1}$  of pairwise disjoint balls.

\mk

Next, let $k_1$ be the smallest integer such that $k_1 >k_0$ and $\mathcal{G}^{(k_1 )}$ is non empty, take an arbitrary  $L_{0}^{(1)}\in\mathcal{G}^{(k_1)}$. It is trivial to check that the family $\mathcal{G}^{(k_0)}\cup \mathcal{G}^{(k_1)}$ and the ball $L_{0}^{(1)}$ satisfy the conditions of Lemma   \ref{dimconst}. Subsequently,  $L_{0}^{(1)}$   intersects at most $\gamma_{d,v}$ balls of $\mathcal{G}^{(k_0)}\bigcup \mathcal{G}^{(k_1)}$, and there must exist  an integer  $1\leq  i _0 \leq \gamma_{d,v}+1$ such that $L_{0}^{(1)}\cap \bigcup_{L\in\mathcal{F}_{i_0}}L=\emptyset$. As before, we add this ball $L_{0}^{(1)} $ to the family $\mathcal{F}_{i_0} $.

Consider $\widetilde {L } \in \mathcal{F}^{(k_1)}$ such that $\widetilde {L }  \notin \bigcup_{1\leq i\leq \gamma_{d,v}+1}\mathcal{F}_i$ (whenever such a ball exists). The exact same argument shows the existence of an integer $1\leq \widetilde i \leq \gamma_{d,v}+1$ such that  $\widetilde {L } $ intersects at most $ \gamma_{d,v}$ balls of $\mathcal{G}^{(k_0)}\bigcup \mathcal{G}^{(k_1)}$. One adds $\widetilde {L } $ to the family $\mathcal{F}_{\, \widetilde i}$, which remains composed only of pairwise disjoint balls. 

One applies this argument to every ball of $\mathcal{F}^{(k_1)}$, hence finally sorting the balls of $\mathcal{F}^{(k_0)}\cup \mathcal{F}^{(k_1)}$ into $\gamma_{d,v}+1$ families of pairwise disjoint balls, as requested.
\mk

It is now easily seen that one can proceed recursively on $k\geq k_0$,  ending up with  the families $\mathcal{F}_1,...,\mathcal{F}_{\gamma_{d,v}+1}$   fulfilling  the desired properties.
\end{proof} 

We are  now ready to prove Proposition \ref{besimodi}.
\begin{proof} Fix $E\subset [0,1]^{d}$ and  $\mathcal{F}=\left\{B(x, r_{(x)} ): x\in E, r_{(x)} >0 \right\}$.

One applies Besicovich's theorem (i.e. Proposition \ref{besimodi} with $v=1$)  to $\mathcal{F}=\left\{B(x,  r_{(x)} ):x\in E \ r_{(x)} >0 \right\}$. This provides us with a finite set of  families of balls $\mathcal{G}_1 ,...,\mathcal{G}_{\gamma_{d,1}+1}$ composed of pairwise disjoint balls satisfying \eqref{besi}, i.e.  $ 
 E\subset  \bigcup_{1\leq i\leq Q_{\gamma_{d,1}+1}} \bigcup_{L\in \mathcal{G}_i}L.$

For every $1\leq i\leq Q_{\gamma_{d,1}+1}$, one sets $\mathcal{G}^{(v)}_i=\left\{\frac{1}{v}L: L\in\mathcal{G}_i\right\}$, i.e. the sets of balls with same centers as $\mathcal{G}_i$ but with radii multiplied by $v^{-1}>1$. Notice that by construction, $\forall  \, 1\leq i\leq Q_{\gamma_{d,1}+1}$, $\forall \, L \neq L' \in\mathcal{G}_i ^{(v)}$, one has $vL  \cap vL' =\emptyset$.  Hence,   Lemma \ref{fami} yields  $\gamma_{d,v}+1$ sub-families $(\mathcal{G}^{(v)}_{i,j})_{1\leq j \leq \gamma_{d,v}+1}$ of $\mathcal{G}_{i}^{(v)}$ such that: \medskip
\begin{itemize}
\item
$\forall \ 1\leq j\leq \gamma_{d,v}+1$, $\forall \ L  \neq L' \in \mathcal{G}^{(v)}_{i,j}$, one has $L  \cap L' =\emptyset,$\medskip
\item
$ \mathcal{G}_{i}^{(v)}=\bigcup_{1\leq j\leq \gamma_{d,v}+1}\mathcal{G}_{i,j}^{(v)}.$\medskip
\end{itemize}
Finally, we set   for every $  1\leq i\leq Q_{d,1} $ and $1\leq j\leq \gamma_{d,v}+1$ 
$$\mathcal{F}_{i,j}=\left\{vL: L\in\mathcal{G}_{i,j}^{(v)}\right\}  \ \ \mbox { and } \ \ \  \mathcal{F}_{i} = \bigcup_{1\leq j\leq \gamma_{d,v}+1}\mathcal{F}_{i,j}  .$$
These sets verify that:
\begin{itemize}
\item 
$\forall \ 1\leq i\leq Q_{d,1}$, $\forall \ 1\leq j\leq \gamma_{d,v}+1$, $\forall L \neq L' \in\mathcal{F}_{i,j}$, $\frac{1}{v}L \cap \frac{1}{v} L' = \emptyset$ (because the balls of $\mathcal{G}_{i,j}$ are pairwise disjoint),\medskip
\item 
$E\subset \bigcup_{1\leq i\leq Q_{d,1}}\mathcal{G}_i =\bigcup_{1\leq i\leq Q_{d,1}} \ \bigcup_{1\leq j\leq \gamma_{d,v}+1}\mathcal{F}_{i,j}.$\medskip
\end{itemize}
This proves the statement and the fact that $Q_{d,v}=Q_{d,1}.(\gamma_{d,v}+1)$.
\end{proof} 
\subsection{Consequences of the $\mu$-asymptotic covering property} 
\label{seccons}

One  first  shows  that the constant   $C$ in Definition \ref{ac}  can  be  replaced by 1 if infinite subsequences of balls are authorized. In fact Definition \ref{ac} ensures that  any open set can be covered (with respect to  the $\mu$-measure) by disjoint balls $B_n$ of arbitrary large indices.

\begin{lemme} 
\label{covO}
Let   $\mu\in\mathcal{M}(\mathbb{R}^d)$  and $\mathcal{B} =(B_n :=B(x_n ,r_n))_{n\in\mathbb{N}}$ be a $\mu$-a.c sequence of balls of $ \mathbb{R}^d$ with $\lim_{n\to +\infty} r_{n}= 0$.

Then  for every open set $\Omega$ and every integer $g\in\mathbb{N}$, there exists  a subsequence  $(B_{(n)}^{(\Omega)})\subset \left\{B_n \right\}_{n\geq g} $ such that:
\begin{enumerate}
\item 
$\forall \,  n\in\mathbb{N}$, $B_{(n)}^{(\Omega)}\subset \Omega,$
\sk
\item
 $\forall \, 1\leq n_1\neq n_2$, $B_{(n_1)}^{(\Omega)}\cap B_{(n_2)}^{(\Omega)}=\emptyset$,
 \sk
\item
 $\mu\left (\bigcup_{n\geq 1}B_{(n)}^{(\Omega)}\right)= \mu(\Omega).$
\end{enumerate}
In addition, there exists an integer $N_{\Omega}$ such that  for the balls $ (B_{(n)}^{(\Omega)})_{n=1,...,N_\Omega}$, the conditions (1) and (2) are realized, and (3) is replaced by  $\mu\left (\bigcup_{n=1}^{N_\Omega} B_{(n)}^{(\Omega)}\right)\geq \frac{3}{4} \mu(\Omega).$  
\end{lemme}

The last part of Lemma \ref{covO} simply follows from item (3) and the $\sigma$-additivity of $\mu$.

\begin{proof} 
 The idea consists in covering $\Omega$ by pairwise disjoint balls amongst those balls of  $\mathcal{B} $,  such that their union has measure at least $C\mu(\Omega)$, then in covering   the complementary of the union of those balls in $\Omega$ (that is still open) with at least a proportion $C$ of its measure, and so on.

More precisely, this is achieved as follows:

\mk$\bullet$
 {\bf Step 1:} By application of Definition \ref{ac} to $\Omega_0 :=\Omega$ and $g\in \N$, there exists $C>0$ and some integers $g\leq n_1 \leq ...\leq n_{N_{0}}$  so that the family of balls $\mathcal{F}_{0}:=\left\{B_{n_i}:=B^{(0)}_i\right\}_{1\leq i\leq N_{0}}$ is pairwise disjoint and $\mu(\bigcup_{1\leq i\leq N_{0}}B_{n_i})\geq C\mu(\Omega).$

\mk$\bullet$ {\bf Step 2:}  Setting $\Omega_1 =\Omega\setminus \bigcup_{ L\in\mathcal{F}_0}L$, applying    Definition \ref{ac} to $\Omega_1$  with the integer $g$ provides us with  a family $\mathcal{F}_1$ of pairwise disjoint balls $B_{1}^{(1)},...,B_{N_1}^{(1)}\in\left\{B_n \right\}_{n\geq g}$ such that $\forall \ 1\leq i\leq N_1$ $B^{(1)}_i \subset \Omega_1$ and $$ \mu(\bigcup_{1\leq i\leq N_2}B_{i}^{(1)})\geq C\mu(\Omega_1).$$ 

One sets $\mathcal{F}_1 =\mathcal{F}_0 \bigcup \mathcal{F}_1 $. One sees that  
\begin{align*}
\mu \left(\bigcup_{L\in\mathcal{F}_1}L\right )&= \mu\left(\bigcup_{L\in\mathcal{F}_0}L\right)+\mu\left(\bigcup_{L\in\mathcal{G}_1}L\right)\\
& \geq \mu\left(\bigcup_{L\in\mathcal{F}_0}L\right)+C\left(\mu(\Omega)-\mu\left( \bigcup_{L\in\mathcal{F}_0}L\right)\right)\\
&\geq (1-C)\mu\left(\bigcup_{L\in\mathcal{F}_0}L\right)+C\mu\left(\Omega\right) \\
&\geq (C+C(1-C))\mu(\Omega) .
\end{align*}
Observe that the balls of $\mathcal{F}_0$ and $\mathcal{F}_1$ are disjoint by construction.

\mk$\bullet$ {\bf Following steps :} Proceeding recursively, and applying the exact same argument as above, one constructs an  increasing sequence of families $(\mathcal{F}_{i})_{i\in\mathbb{N}}$ and a  decreasing sequence of open sets $\Omega_i$ such that:
\begin{itemize} 
\sk
\item 
$\forall \, i\in\mathbb{N}$,  $L\in \{B_n \}_{n\geq g} $ and $\forall L\in\mathcal{F}_i $, $L\subset \Omega_i\subset \Omega,$

\sk
\item  
$\forall  \,i\in\mathbb{N}$, $\forall L \neq L'\in\mathcal{F}_i $, $L  \cap L' =\emptyset$, 

\sk
\item  
$\forall  \, i\neq j\in\mathbb{N}$, $\forall L \in\mathcal{F}_i $ and $\forall L \in\mathcal{F}_j $, $L \cap L' =\emptyset$, 

\sk
\item $\forall  \, i\in\mathbb{N}$,  
$\mu\left(\bigcup_{L\in\mathcal{F}_i}L\right)\geq \mu(\Omega) \sum_{1\leq k\leq i}C(1-C)^{k-1}.$ 
\end{itemize}

Finally, setting $\mathcal{F}=\bigcup_{i\in\mathbb{N}}\mathcal{F}_i$, one sees that  $\mathcal{F}$ is constituted by pairwise disjoint balls chosen amongst $\left\{B_n\right\}_{n\geq g}$   satisfying 

\begin{equation}
\label{usefulmath}
\mu(\Omega)\geq \mu\left(\bigcup_{L\in\mathcal{F}}L\right )\geq \mu(\Omega)\sum_{k\geq 1}C(1-C)^{k-1}=\mu(\Omega),
\end{equation}
so that $\mathcal{F}$ fulfills the conditions of Lemma \ref{covO}.    
\end{proof}

An easy consequence is the following.
\begin{corollary}
\label{covBor}
Let $\mu \in\mathcal{M}(\mathbb{R}^d)$ and $(B_n )_{n\in\mathbb{N}}$ be a $\mu$-a.c sequence of balls. Then for any Borel set $E$, for any $g\in\mathbb{N}$, there exists a sub-sequence of balls $(B_{(n)}^{(E)})\subset \left\{B_n\right\}_{n\geq g}$ such that:
\begin{enumerate}
\item $\forall 1\leq n_1 \neq n_2$, $B_{(n_1)}^{(E)}\cap B_{(n_2)}^{(E)}=\emptyset,$\mk
\item $\mu\Big(\bigcup_{n\in\mathbb{N}}B_{(n)}^{(E)}\cap E\Big)=\mu(E),$\mk
\item $\mu\Big(\bigcup_{n\in\mathbb{N}}B_{(n)}^{(E)}\Big)\leq \mu(E)+\varepsilon,$
\end{enumerate} 
\end{corollary}
\begin{proof}
By outer regularity, there exists an open set $\Omega$ such that $E\subset \Omega$ and $m(\Omega)\leq \mu(E)+\varepsilon.$ Applying Lemma \ref{covO} to $\Omega$, the sequence $(B_n)_{n\in\mathbb{N}}$ fulfills the condition of Corollary \ref{covBor}.
\end{proof}



\subsection{Proof of Theorem \ref{equiac}} 

\label{secequiac}

{\bf (1)} Assume first that    $\mathcal{B} = (B_n)_{n\in\mathbb{N}}$ is $\mu$-a.c, and let us prove that $\mu(\limsup_{n\rightarrow+\infty}B_n)=1$.

For every $g\in\mathbb{N}$, applying Lemma \ref{covO}, there exists a  sub-family of balls, $\mathcal{F}_g \subset \left\{B_{n}\right\}_{n\geq g}$ such that $\mu(\bigcup_{L\in\mathcal{F}_g}L)=\mu(\mathbb{R}^d)=1.$  In particular, $\mu(\bigcup_{n\geq g}B_n)=1$ for every $g\geq 1$, and   $\mu(\limsup_{n\rightarrow+\infty}B_n)=\mu(\bigcap_{g\geq 1}\bigcup_{n\geq g}B_n)=1$.   
\medskip

{\bf (2)} Suppose next  that there exists $v<1$ such that $\mu(\limsup_{n\rightarrow+\infty}v B_n)=1$, and let us show that  $\mathcal{B} $ is $\mu$-a.c.

Let $\Omega$ be an open set in $\mathbb{R}^d$. Our goal is to find a constant $C$ such that the conditions of Definition \ref{ac} are realized.

%
%
%

\mk

Let  $E= \Omega \cap  \limsup_{n\rightarrow+\infty}vB_n$. 
For every   $y\in E$, consider an integer $n_y\geq g$ large enough so that $y\in vB_{n_y}$ and $B(y,2r_{n_y} )\subset \Omega.$  This is possible since $\lim_{n\to +\infty} r_n=0$.

Since $y\in vB_{n_y}$, one has 
\begin{equation}
\label{eq11}
  B(y,(1-v)r_{n_y})   \subset B_{n_y} \subset B(y, (1+v) r_{n_y}) \subset B(y,2r_{n_y}),
\end{equation}
 and the family     $\mathcal{F}= \left\{  B(y, (1-v) r_{n_y}): y\in E\right\}$ covers $E $ by balls   centered on $E$.

Applying Proposition \ref{besimodi} with constant $v'= \frac{1-v}{2} <1$  allows to extract from $\mathcal{F}$  finite or countable sub-families   $\mathcal{F}_1,...,\mathcal{F}_{Q_{d,v'}}$  such that: 
\begin{itemize}

\item
$\forall 1\leq i\leq Q_{d,v'}$, $L\neq L'\in\mathcal{F}_i$, one has $ \frac{1}{v'} L \cap  \frac{1}{v'} L'=\emptyset.$\medskip

\item 
$E$ is covered by the families $\mathcal{F}_i$, i.e.  \eqref{besi} holds true.
\end{itemize}

Now, $\mu(\Omega) =\mu(E) \leq  \mu \left ( \bigcup_{i=1}^{ Q_{d,v'}} \bigcup_{L\in\mathcal{F}_{i}} L \right ) $.
 There must exist  $1\leq i_0 \leq Q_{d, v'}$ such  that 
$$
\mu \left (\bigcup_{L\in\mathcal{F}_{i_0}}L \right)\geq \frac{1}{Q_{d,v'}}\mu(E )= \frac{1}{Q_{d, v'}}\mu(\Omega).
$$
 
 There exist   $L_1$, $L_2$, ... $L_N$   balls of  $\mathcal{F}_{i_0}$ such that
$$\mu\left (\bigcup_{1\leq k\leq N}L_{k}  \right )\geq\frac{1}{2 Q_{d,v'}}\mu(\Omega),
$$
Notice the following facts:
\begin{itemize}
\item
$\forall 1\leq i\leq Q_{d,v'}$,  every $L \in\mathcal{F}_i$ is naturally associated with some $y \in E$ and some ball $B_{n_y}$,  with $L\subset B_{n_y} \subset \Omega$,
\item 
$\forall 1\leq i\leq Q_{d,v'}$,  if $L \in\mathcal{F}_i$ is   associated with   $y \in E$ and     $B_{n_y}$  and $L' \in\mathcal{F}_i$ is   associated with   $y' \in E$ and   $B_{n_{y'}}$, then    $ \frac{1}{v'} L \cap  \frac{1}{v'} L'=\emptyset$ implies by \eqref{eq11} that  $B_{n_{y}} \cap B_{n_{y'}}=\emptyset$.

\end{itemize}
The first fact implies that   there exist $N$ integers $n_1$, ..., $n_N$ such that  $B_{n_k}\subset \Omega$ and 
$$
\mu\left (\bigcup_{1\leq k\leq N} B_{n_k}   \right )\geq\frac{1}{2 Q_{d,v'}}\mu(\Omega),
$$
The second fact implies that these balls $B_{n_k} $, $k=1,..., N$ are pairwise disjoint.

This exactly proves that $\mathcal{B}$ is $\mu$-a.c.

\subsection{A version of Borel-Cantelli Lemma}
\label{secBC}
In this manuscript, one mainly focuses on establishing Hausdorff dimension of limsup sets knowing that a certain limsup set of balls has full measure. In many situation, proving that those limsup sets have full measure is straightforward. When it is not, it is convenient to have a tool at our disposal to be able to determine whether or not it is the case. In the case where the measure involved is doubling is treated by Beresnevich-Velani.
\begin{theoreme}[\cite{BVBC}]
Let $\mu\in\mathcal{M}(\mathbb{R}^d)$ be a doubling measure and $(B_n)_{n\in\mathbb{N}}$ a sequence of balls centered in $\supp(\mu)$ such that $\vert B_n \vert\to 0.$ Then $\mu(\limsup_{n\in\mathbb{N}}B_n)=1\Leftrightarrow \exists C>1$ such that for any open ball $B$ centered on $\supp(\mu),$ there exists a sub-sequence $(L_{B,n})_{n\in\mathbb{N}}$ of $(B_n)_{n\in\mathbb{N}}$ satisfying:
\begin{itemize}
\item[•] $L_{B,n}\subset B,$\sk
\item[•] $\sum_{n=0}^{+\infty}\mu(L_{B,n})=+\infty,$\sk
\item[•] for infinitely many $Q  \in\mathbb{N},$ 
\begin{equation}
\sum_{s,t=1}^{Q}\mu(L_{B,s}\cap L_{B,t})\leq \frac{C}{\mu(B)}\left(\sum_{n=1}^{Q}\mu(L_{B,n})\right)^2.
\end{equation}
\end{itemize}
\end{theoreme}

Thanks to Theorem \ref{equiac}, one can complete this Theorem and remove the doubling assumption.

\begin{proposition}
\label{BCmodi}
Let $(B_n)_{n\in\mathbb{N}}$ be a sequence of closed balls satisfying $\vert B_n \vert\to 0$ and $ \mu\in\mathcal{M}(\mathbb{R}^d)$ be a probability measure. 
\begin{itemize}
\item[(A):] Assume that $(B_n)_{n\in\mathbb{N}}$ is $\mu$-a.c, then, there exists $C>1$ such that for any open ball $B$, there exists a sub-sequence of $(B_n)_{n\in\mathbb{N}}$,  $(L_{n,B})_{n\in\mathbb{N}}$ satisfying, for any $n\in\mathbb{N}$, $L_{B,n} \subset B$ and 
\begin{equation}
\label{eq1}
\sum_{n \geq 0}\mu(L_{B,n})=+\infty 
\end{equation}
and for infinitely many Q,
\begin{equation}
\label{eq2}
\sum_{s,t=1}^{Q}\mu(L_{B,s}\cap L_{B,t})\leq \frac{C}{\mu(B)}\left(\sum_{n=1}^{Q}\mu(L_{B,n})\right)^2.
\end{equation}
\item[(B):] Assume that there exists $C>1$ such that for any open ball $B$, there exists a sub-sequence of $(B_n)_{n\in\mathbb{N}}$ $(L_{n,B})_{n\in\mathbb{N}}$ with, for any $n\in\mathbb{N},$ $L_{n,B}\subset B$, satisfying  \eqref{eq1} and \eqref{eq2}, then $\mu(\limsup_{n\rightarrow+\infty}B_n)=1,$ so that, for any $\kappa >1,$ $(\kappa B_n)_{n\in\mathbb{N}}$ is $\mu$-a.c.
\end{itemize} 
\end{proposition}
\begin{proof}
 Item $A$ is proved in \cite{BVBC} (this part of the proof does not use the doubling property of the measure in \cite{BVBC}). Moreover, it is also proved in \cite{BVBC} that, if there exists $C>0$ such that for any open ball $B$, there exists a sub-sequence of $(B_n)_{n\in\mathbb{N}}$ $(L_{n,B})_{n\in\mathbb{N}}$ with, for any $n\in\mathbb{N},$ $L_{n,B}\subset B$, satisfying  \eqref{eq1} and \eqref{eq2}, then 
 $$\mu(\limsup_{n\to +\infty}B_n \cap B)\geq \frac{1}{C}\mu(B).$$
 The following lemma combined with Theorem \ref{equiac} finishes the proof of Proposition \ref{BCmodi}.
 \begin{lemme}
 \label{muboule}
 Let $E\subset \mathbb{R}^d$. Assume that there exists $0<c<1$ such that for any open ball $B$, $\mu(E\cap B)\geq c\mu(B),$ then $\mu(E)=1.$
 \end{lemme}  
 \begin{proof}
 Assume that $\mu(E)<1$ and set $A=\mathbb{R}^d \setminus E.$ By hypothesis, $\mu(A)>0.$ 
 
Let us recall the following density-lemma (which holds in metric sapces in which Besicovitch's theorem holds).
\begin{lemme}\cite{Be}
 \label{densibesi}
 {Let $m\in\mathcal{M}(\R^d)$, $0<c<1$ and $A$ be a Borel set with $m(A)>0.$ For every $r>0$, set}
 \begin{equation} 
 \label{defniv} 
 {A(r) =\left\{x\in A \ : \ \forall \tilde{r}\leq r, \ m(B(x,\tilde{r})\cap A)\geq c.m(B(x,\tilde{r}))\right\}} 
 \end{equation}
 {Then }
 \begin{equation}
{ m\left(\bigcup_{r>0}A(r)\right)=m(A).}
 \end{equation}
 \end{lemme} 
 
 By Lemma \ref{densibesi}, there exists an open ball $B$ such that $\mu(B)>0$ and 
 $$\mu(B\cap A)\geq (1-\frac{c}{2})\mu(B).$$
 This yields
 \begin{align*}
 \mu(E\cap B\cap A)&=\mu(E\cap B)+\mu(A\cap B)-\mu\left((E\cap B)\cup(A\cap B)\right)\\
 &\geq (c+1-\frac{c}{2}-1)\mu(B)=\frac{c}{2}\mu(B)>0,
 \end{align*}
 which implies $\mu(E\cap A)>0$, which is a contradiction.
 \end{proof}
Taking $c=\frac{1}{C}$ and applying Lemma \ref{muboule} finishes the proof of Proposition \ref{BCmodi}.
\end{proof}
\begin{remark}
A version of Proposition \ref{BCmodi} might also be useful in more general metric spaces. The only geometric  property we used to prove Proposition \ref{BCmodi} is actually Proposition \ref{besimodi} (which also implies Lemma \ref{densibesi}), so that Proposition \ref{BCmodi} actually holds in any direction-limited spaces as defined in \cite{Fed}.
\end{remark}

\section{proof of Theorem \ref{theoremextra}}

\label{sec-extrac} 

The following section is dedicated to the study of the properties one can ask an $\mu$-a.c sequence $(B_n)_{n \in\mathbb{N}}$ to verify, up to an $\mu$-a.c extraction. 

The concept of conditioned ubiquity was introduced by  Barral and Seuret in \cite{BS3}. It consists in asking the balls of the sequence $(B_n)$ to verify some specific properties with respect to the measure $\mu$. When investigating the Hausdorff dimension of some sets $(U_n)_{n\in\mathbb{N}}$, where $U_n \subset B_n$, in practical cases (when the measure carries some self-similarity), it turns out that when a lower-bound is found for $(U_n)_{n\in\mathbb{N}}$ using the fact that the sequence $(B_n)_{n\in\mathbb{N}}$ is of limsup of full $\mu$-measure, it is often quite easy to prove that $\limsup_{n\rightarrow+\infty}U_n$ has precisely the expected measure provided that the sequence $(B_n)$ verifies some specific properties with respect to $\mu$.

Note that in full generality, understanding the  optimality  of a bound as mentioned above, means  understanding very finely the behavior of the measure $\mu$ on the sets $U_n$ (the sequence $(B_n)$ being $\mu$-a.c). It will be proved in this article that, under mild  conditions on the sequence $(B_n)_{n\in\mathbb{N}}$, it is always  possible to give a natural upper-bound for $\dim_H (\limsup_{n\rightarrow+\infty}U_n)$. This upper-bound turns out to be optimal when the measure carries enough self-similarity (in particular it works for quasi-Bernoulli measures or fully supported self-similar measures).

\textbf{In this section, the balls $(B_n)_{n\in\mathbb{N}}$ are supposed to be pairwise distinct and such that $\vert B_n\vert \underset{n\rightarrow+\infty}{\rightarrow}0$.}

\subsection{Extraction of  weakly redundant $\mu$-a.c subsequences}
\label{secnest1}

The main result of this section is stated here.

\begin{proposition}
\label{extrawr}
Let $\mu \in \mathcal{M}(\mathbb{R}^d)$ and $(B_n)_{n\in\mathbb{N}}$ be a $\mu$-a.c sequence of balls $\lim_{n\rightarrow+\infty}\vert B_n \vert=0.$ There exists a subsequence $(B_{\psi(n)})_{n\in\mathbb{N}}$ of $(B_n)_{n\in\mathbb{N}}$ which is weakly redundant and $\mu$-a.c.
\end{proposition}

\begin{proof}
Let $g_k \in\mathbb{N}$ be large enough so that $\forall n\geq g_k$, $\vert B_n \vert  \leq 2^{-k}.$  By Proposition \ref{covO}, applied with the sequence $(B_n)_{n\in\mathbb{N}}$, $\Omega=\mathbb{R}^d$ for any $k \in \mathbb{N}$, there exists a sub-sequence $(B_{(n,k)})$ of $\left\{B_n\right\}_{n\geq g_k}^{\mathbb{N}}$ satisfying  
\begin{enumerate}
\item $\forall 1\leq n_1 \neq n_2$, $B_{(n_1 ,k)}\cap B_{(n_2 ,k)}=\emptyset,$\mk
\item $\mu\Big(\bigcup_{n\in\mathbb{N}}B_{(n,k)}\Big)=1.$\mk
\end{enumerate} 
 Define $\mathcal{B}_{\psi}=(B_{\psi(n)})_{n\in\mathbb{N}}$ as the sub-sequence of balls corresponding to  $\bigcup_{k\in\mathbb{N}}\left\{B_{(n,k)}\right\}_{n\in\mathbb{N}}.$

Since  the following inclusion holds
\begin{equation}
\label{reuint}
\bigcap_{k\in\mathbb{N}}\bigcup_{n\in\mathbb{N}}B_{(n,k)}\subset \limsup_{n\to+\infty}B_{\psi(n)},
\end{equation}
 by item $(2)$ one has
$$\mu(\limsup_{n\to+\infty}B_{\psi(n)})=1.$$

Note that,  for all $k\in\mathbb{N}$, all $B\in \left\{B_{(n,k)}\right\}_{n\in\mathbb{N}}$, $\vert B\vert \leq 2^{-k}$. Following the notation of Definition \ref{wr}, for any $k\in\mathbb{N}$, $\mathcal{T}_k (\mathcal{B}_{\psi})$ can contain only balls of the sequence  of the $k$ first families $\left\{B_{(n,k)}\right\}_{n\in\mathbb{N}}$, which are composed of pairwise disjoint balls. This proves that $\mathcal{T}_k (\mathcal{B}_{\psi})$ can be sorted in at most $k+1$ families of  pairwise disjoint  balls. In particular, $\mathcal{B}_{\psi}$ is weakly redundant.  

It remains to show that $(B_{\psi(n)})_{n\in\mathbb{N}}$ is $\mu$-a.c.

Let $\Omega$ be an open set and $g\in\mathbb{N}$. One will extract from $\mathcal{B}_{\psi}$ a finite number of balls satisfying the condition of Definition \ref{ac}. 

 There exists $k_0 $ so large that, $\mu\big(\left\{x\in \Omega(x,2^{-k_0+1})\right\}\big)\geq \frac{3\mu(\Omega)}{4}$ for any $k\geq k_0$, 
   $\left\{B_{(n,k)}\right\}_{n\in\mathbb{N}}\subset \left\{B_{\psi(n)}\right\}_{n\geq g}$  and $\mu(\limsup_{n\rightarrow+\infty}B_{\psi(n)}\cap \Omega )\geq \frac{3\mu(\Omega)}{4}$. Setting
 $$\widehat{E}=\left\{x\in\limsup_{n\rightarrow+\infty}B_{\psi(n)}\cap \Omega: B(x,2^{-k_0 +1})\subset \Omega\right\},$$ 
 it holds that
$$\mu(\widehat{E})\geq \frac{1}{2}\mu(\Omega).$$

Recalling  \eqref{reuint}, for every $x\in \widehat{E}$, consider $B_{x}$, the  ball of $\left\{B_{(n,k_0)}\right\}_{n\in\mathbb{N}}$ containing $x$. Note that, because for  $B\in\left\{B_{(n,k_0)}\right\}_{n\in\mathbb{N}}$, $\vert B\vert \leq 2^{-k_0}$, one has $B_x \subset B(x,2^{-k_0 +1})\subset \Omega .$ Set $\mathcal{F}_1 =\left\{B_x :x\in\widehat{E}\right\}.$ The set  $\mathcal{F}_1$ is composed of pairwise disjoint balls (by item $(1)$ above) of $\left\{B_{\psi(n)}\right\}_{n\geq g} $ included in $\Omega$ and such that

\begin{equation}
\mu\Big(\bigcup_{L\in\mathcal{F}_1}L\Big)\geq \mu(\widehat{E})\geq \frac{1}{2}\mu(\Omega).
\end{equation}

Using the $\sigma$-additivity of $\mu$ concludes the proof.
\end{proof}
 \medskip
 
\subsection{Extraction of sub-sequences of balls with conditioned measure} 
 \label{seccondi}
 
 Let $\mu \in\mathcal{M}(\mathbb{R}^d)$ and $(B_n)_{n\in\mathbb{N}}$ be an $\mu$-a.c sequence of balls. 
 
 This part aims to understand what condition can be assumed about the measure of the ball of the sequence $(B_n)_{n\in\mathbb{N}}$ in general under the $\mu$-a.c condition.

More precisely, item $(2)$ of Theorem \ref{theoremextra} is proved.

\begin{proposition}
\label{mainextraprop}
Let $\mu\in\mathcal{M}(\mathbb{R}^d)$. For any sequence of balls $(B_n)_{n\in\mathbb{N}}$ satisfying $\mu(\limsup_{n\rightarrow+\infty}vB_n)=1$ for some $0<v<1$, there exists an $\mu$-a.c sub-sequence $(B_{\phi(n)})_{n\in\mathbb{N}}$ verifying
\begin{equation*}
\label{condimesu1}
  \underline{\dim}_H (\mu)\leq \liminf_{n\rightarrow+\infty}\frac{\log \mu(B_{\phi(n)})}{\log\vert B_{\phi(n)}\vert}\leq \limsup_{n\rightarrow+\infty}\frac{\log \mu(B_{\phi(n)})}{\log \vert B_{\phi(n)}\vert}\leq \overline{\dim}_P (\mu).
  \end{equation*}
\end{proposition}

 \begin{remarque}
For the left part of \eqref{condimesu1}, the proof actually only uses the fact that $\mu (\limsup_{n\rightarrow+\infty}B_n)=1.$
\end{remarque}

Let us introduce some useful sets to prove Lemma \ref{gscainf} and Lemma \ref{gscaup}, which are key in order to prove \eqref{condimesu1}.


{
\begin{definition} 
\label{emudef}
 Let $0\leq \alpha\leq\gamma$ be  real numbers, $\mu\in\mathcal{M}(\mathbb{R}^d)$, and $\varepsilon,\rho>0$ two positive real numbers. Then define 
\begin{equation}
\label{emuhat}
 {{E}_{\mu}^{[\alpha,\gamma],\rho,\varepsilon}=\left\{x\in\mathbb{R}^d : \ \underline\dim (\mu,x)\in[\alpha,\gamma] \text{ and }\forall r\leq \rho, \  \mu(B(x,r))\leq r^{\underline{\dim} (\mu,x)-\varepsilon}\right\}},
\end{equation} 
\begin{equation}
 \label{fmu}
 F_{\mu}^{[\alpha,\beta],\rho,\varepsilon}=\left\{x\in \mathbb{R}^d :  \overline{\dim} (\mu,x)\in [\alpha, \beta]\text{ and }\forall r<\rho, \ \mu(B(x,r))\geq r^{\overline{\dim} (\mu,x)+\varepsilon}\right\}.
\end{equation} 
and
\begin{align} 
\label{emut}
 E_{\mu}^{[\alpha,\gamma],\varepsilon}   =\bigcup_{n\geq 1}E_{\mu}^{[\alpha,\gamma],\frac{1}{n},\varepsilon}   \ \ \ \mbox{ and }\ \ \ \  F_{\mu}^{[\alpha,\gamma],\varepsilon}  =\bigcup_{n\geq 1}F_{\mu}^{[\alpha,\gamma],\frac{1}{n},\varepsilon}.
\end{align}
 \end{definition}} 
{The following statements are easily deduced from Definition \ref{dim}. }


\begin{proposition}
\label{résu}
For every $\mu\in\mathcal{M}(\mathbb{R}^d)$, $\rho>0$, every $0\leq \alpha \leq \gamma$ and  $\varepsilon>0$,  
\begin{align*}
\mu(E_{\mu}^{[\alpha,\gamma],\varepsilon}) & =\mu(\left\{x:\underline{\dim} (\mu,x)\in [\alpha,\gamma]\right\})\\
\mu(F_{\mu}^{[\alpha,\gamma],\varepsilon})& =\mu(\left\{x:\overline{\dim} (\mu,x)\in [\alpha,\gamma]\right\})
\end{align*}
 and
  \begin{align}
 \label{mlecoro}
 E_{\mu}^{[\alpha,\gamma],\rho,\varepsilon} & \subset \left\{x\in \mathbb{R}^d : \ \forall r\leq \rho, \  \mu(B(x,r))\leq r^{\alpha-\varepsilon}\right\}\\
\nonumber  F_{\mu}^{[\alpha,\gamma],\rho,\varepsilon} & \subset \left\{x\in \mathbb{R}^d : \ \forall r\leq \rho, \ \mu(B(x,r))\geq r^{\gamma+\varepsilon}\right\}.
\end{align}
Furthermore, for $\alpha_1=\underline{\dim}_H (\mu)$ and $\gamma_1=\mbox{supess}_{\mu}(\underline{\dim} (\mu,x))$, one has
\begin{equation}
\label{mle}
\mu(E_{\mu}^{[\alpha_1,\gamma_1],\varepsilon})=1.
\end{equation}  
Similarly, for  $\alpha_2=\mbox{infess}_{\mu} (\overline{\dim} (\mu,x))$ and $\gamma_2=\overline{\dim}_P (\mu)$, one has 
\begin{equation}
\label{mleup}
\mu\left(F_{\mu}^{[\alpha_2,\gamma_2],\varepsilon}\right)=1.
\end{equation}

\end{proposition}

 \begin{proof}
{ For any $x \in \mathbb{R}^d$, for any $\varepsilon>0$, there exists $r_x >0$ such that, $\forall r\leq r_x $, $r^{\overline{\dim} (\mu,x)+\varepsilon}\leq \mu(B(x,r))\leq r^{\underline{\dim} (\mu,x)-\varepsilon}.$ This implies 
\begin{equation*}
F_{\mu}^{[\alpha,\gamma],\rho,\varepsilon}\subset \left\{x\in \mathbb{R}^d : \ \forall r\leq \rho, \ \mu(B(x,r))\geq r^{\gamma+\varepsilon}\right\},
\end{equation*} 
 \begin{equation*} 
 E_{\mu}^{[\alpha,\gamma],\rho,\varepsilon}\subset \left\{x\in\mathbb{R}^d : \ \forall r\leq \rho, \ \mu(B(x,r))\leq r^{\alpha-\varepsilon}\right\},
 \end{equation*}}
and
 $$E_{\mu}^{[\alpha,\gamma],\varepsilon}=\left\{x:\underline{\dim} (\mu,x)\in [\alpha,\gamma]\right\}  \text{ and } F_{\mu}^{[\alpha,\gamma],\varepsilon}=\left\{x:\overline{\dim} (\mu,x)\in [\alpha,\gamma]\right\}.$$
 {
 Since $$\mu([\mbox{infess}_{\mu} (\underline{\dim} (\mu,x)),\mbox{supess}_{\mu}(\underline{\dim} (\mu,x))])=1$$ and $$\mu([\mbox{infess}_{ \mu} (\overline{\dim} (\mu,x)),\mbox{supess}_{\mu}(\overline{\dim} (\mu,x))])=1,$$}
 { recalling Definition \ref{dim}, it holds that, following the notation of Proposition \ref{résu},
 $$\mu\left(E_{\mu}^{[\alpha_1,\gamma_1],\varepsilon}\right)=\mu\left(F_{\mu}^{[\alpha_2,\gamma_2],\varepsilon}\right)=1.$$ }
 \end{proof}

Before showing Proposition \ref{mainextraprop}, let us start by the two following Lemmas \ref{gscainf} and \ref{gscaup}. The first one will be used to prove the left part of the inequality \eqref{condimesu} while the second one will be useful to prove the right part.
\begin{lemme} 
\label{gscainf}
Let $\mu \in\mathcal{M}(\mathbb{R}^d)$ and $\mathcal{B} =(B_n :=B(x_n ,r_n))_{n\in\mathbb{N}}$ be a  $\mu$-a.c sequence of balls of $ \mathbb{R}^d$ with $\lim_{n\to +\infty} r_{n}= 0$.

 For any $\varepsilon>0$, there exists a  $\mu$-a.c  subsequence $(B_{\phi(n)})_{n\in\mathbb{N}}$ of $\mathcal{B} $ such that for every $n\in\mathbb{N}$, $\mu(B_{\phi(n)})\leq (r_{\phi(n)})^{ \underline{\dim}_H(\mu)-\varepsilon}.$
\end{lemme}

\begin{proof} 
Set $\alpha= \underline{\dim}_H (\mu)$ and $\gamma=suppess_{\mu}(\underline{\dim} (\mu,x)).$

Let $\Omega$ be an open set and $\ep>0$. By  \eqref{mle},  $\mu(E^{[\alpha,\gamma],\frac{\varepsilon }{2}}_{\mu })=1$ and $\mu  ( \Omega\cap  E^{[\alpha,\gamma],\frac{\varepsilon }{2}}_{\mu}) = \mu(\Omega)$.

For every $x\in \Omega\cap  E^{[\alpha,\gamma],\frac{\varepsilon }{2}}_{\mu}$,  there exists $r_x>0 $ such that $B(x,r_x) \subset \Omega$ and $x\in E_{\mu}^{[\alpha,\gamma],r_x ,\frac{\varepsilon}{2}}.$ 

Recall  \eqref{emut} and that the sets $E_{\mu}^{[\alpha,\gamma], \rho ,\varepsilon}$ are non-increasing in $\rho$. In particular there exists  $\rho_{\Omega}>0$ such that the set $E_\Omega:=\left\{x\in \Omega\cap E^{[\alpha,\gamma],\rho_{\Omega},\frac{\varepsilon }{2}}_{\mu} : \ r_x\leq \rho_\Omega \right\}$ verifies
\begin{equation}
\label{muomegainf}
\mu(E_\Omega ) \geq \frac{3\mu(\Omega)}{4}.
\end{equation}

 Let $g\in\mathbb{N}$. Applying Lemma \ref{covO} 
to $\Omega$, the sequence $(B_n)$ and the measure $m$, there exists $N_\Omega$ as well as  $g\leq n_{1}\leq ...  \leq n_{N_{\Omega}}$ verifying:
\begin{enumerate}
\item   for every $  1\leq i\neq j\leq N_\Omega,$   $B_{n_i}\cap B_{n_j}=\emptyset$, 
\item  for every $  1\leq i\leq N_\Omega$, $ 2r_{n_i} \leq \rho_\Omega $ and $2^{\alpha -\frac{\varepsilon}{2}}\leq r_{n_i}^{-\frac{\varepsilon}{2}}$, 
\item $\mu(\bigcup_{1\leq i\leq N_\Omega}B_{n_i})\geq \frac{\mu(\Omega)}{2}.$ 
\end{enumerate}
We may assume that $\mu(B_{n_i}) >0$ for every $i$, otherwise $B_{n_i}$ does not play any role.

Item $(3)$ together with \eqref{muomegainf} implies that
$$\mu\left (\bigcup_{1\leq i\leq N_\Omega}B_{n_i}\cap E_\Omega \right)\geq \frac{\mu(\Omega)}{4}.$$

 Furthermore, for every $1\leq i\leq N_\Omega$ verifying $B_{n_i}\cap E_\Omega \neq \emptyset $, it holds that $0<\mu(B_{n_i}) \leq (r_{n_i})^{\alpha-\varepsilon }.$ Indeed, let $x\in B_{n_i}\cap E_\Omega$. By item (2), $B_{n_i} \subset B(x,2r_{n_i})$, and by \eqref{emuhat}  , item $(2)$, and  \eqref{mlecoro}, it holds that
$$\mu(B_{n_i})\leq \mu(B(x,2r_{n_i}))\leq (2r_{n_i})^{\alpha-\frac{\varepsilon}{2}}\leq (r_{n_i})^{\alpha-\varepsilon }.$$

Writing $\mathcal{B}'=\left\{B_n : \mu(B_{n})\leq r_n ^{\alpha-\varepsilon }\right\}$, the argument above shows that only balls of   $\mathcal{B}'$ have been used to cover $\Omega$ . This is satisfied for every open set $\Omega$, so that  $\mathcal{B}'$ is a sub-sequence of $\mathcal{B}$  satisfying the condition  of Definition \ref{ac}, which concludes the proof of Lemma \ref{gscainf}.
 \end{proof}

 \begin{lemme} 
\label{gscaup}
Let  $\mu \in\mathcal{M}(\mathbb{R}^d)$,  $v<1$ and $\mathcal{B} =(B_n :=B(x_n ,r_n))_{n\in\mathbb{N}}$ a sequence of balls of $ \mathbb{R}^d$ verifying $\mu(\limsup_{n\to +\infty}vB_n)=1$.

For all $\varepsilon>0$, there exists a sub-sequence $(B_{\phi(n)})_{n\in\mathbb{N}}$  of  $\mathcal{B} $ as well as $0<v^{\prime}<1$ such that $\mu(\limsup_{n\rightarrow+\infty}v^{\prime}B_{\phi(n)})=1$  and for all $n\in\mathbb{N}$, one has $\mu(B_{\phi(n)})\geq (r_{\phi(n)})^{\overline{\dim}_H (\mu)+\varepsilon}$.    
\end{lemme}

\begin{remarque}
The sequence $(B_{\phi(n)})_{n\in\mathbb{N}}$ found in Lemma \ref{gscaup} is in particular $\mu$-a.c by Theorem \ref{equiac}.
\end{remarque}



\begin{proof} 
Let $\alpha =\mbox{infess}_{\mu} (\overline{\dim} (\mu,x))$ and $\gamma=\overline{\dim}_P (\mu).$
Let $\ep>0$ and $v< v^{\prime}<1$.

By   \eqref{mleup} and Theorem \ref{equiac}, $\mu(\limsup_{n\rightarrow+\infty}vB_n \cap F_{\mu}^{[\alpha,\gamma],\frac{3\varepsilon}{2}})=1.$ For all $x\in\limsup_{n\rightarrow+\infty}vB_n \cap F^{[\alpha,\gamma],\frac{3\varepsilon}{2}}_{\mu}$, there exists $r_x>0$ small enough so that
\begin{equation}
\label{randcond}
r_x ^{\frac{\varepsilon}{2}}\leq {(v^{\prime}-v)}^{\gamma +\frac{3\varepsilon}{2}} \text{ and } \forall 0<r\leq r_x, \  \mu(B(x,r))\geq r^{\gamma +\frac{3\varepsilon}{2}}.
\end{equation}

 Since $x\in \limsup_{n\rightarrow+\infty}vB_n$, for all $n\in\mathbb{N}$, there exists $n_x \geq n$ such that $x\in vB_{n_x}$ and $ (v^{\prime}-v)r_{n_x} \leq r_x .$ Note that $B(x,(v^{\prime}-v)r_{n_x})\subset v^{\prime}B_{n_x}.$ This implies the following inequalities:
 $$\mu(B_{n_x})\geq \mu(v^{\prime}B_{n_x})\geq \mu(B(x,(v^{\prime}-v)r_{n_x})\geq ((v^{\prime}-v)r_{n_x})^{\gamma+ \frac{3\varepsilon}{2}}\geq r_{n_x}^{\gamma+2\varepsilon}.$$  
Set $\mathcal{B}_{\gamma,2\varepsilon}=\left\{B_n : \mu(B_n)\geq r_n ^{\gamma+2\varepsilon}\right\}$.
One just showed that 
$$\limsup_{n\rightarrow+\infty}vB_n \cap F^{[\alpha,\gamma],\frac{\varepsilon}{2}}\subset\limsup_{B\in\mathcal{B}_{\gamma,2\varepsilon}}v^{\prime}B.$$ 
This proves that $\mu(\limsup_{B\in\mathcal{B}_{\gamma,2\varepsilon}}v^{\prime}B)=1$.

Since $\varepsilon>0$ was arbitrary, the results also holds with $\frac{\varepsilon}{2}$, which proves Lemma \ref{gscaup}. 
 \end{proof}

We are now ready to prove Proposition \ref{mainextraprop}.

\begin{proof}
Set $\alpha=\underline{\dim}_H (\mu)$ and $\beta=\overline{\dim}_H (\mu).$

Let us fix $(\varepsilon_n)_{n\in\mathbb{N}}\in(\mathbb{R}^{*+})^{\mathbb{N}}$ verifying $\lim_{n\rightarrow+\infty}\varepsilon_n =0.$

The strategy of the proof consists in constructing recursively coverings of the cube $\mathbb{R}^d$ by using Lemma \ref{gscainf} and Lemma \ref{gscaup} and  a  diagonal argument (on the choice of $\varepsilon$) at each step.

More precisely, at step 1, one will build a family of balls $(\mathcal{F}_{1,i})_{i\in\mathbb{N}}$ verifying:
\begin{itemize}
\item[•] for all $i,j\geq 1$, $\forall L\in\mathcal{F}_{1,i}$, $\forall L^{\prime}\in\mathcal{F}_{1,j}$ such that $L\neq L^{\prime},$ one has $L\cap L^{\prime}=\emptyset,$\medskip
\item[•] for all $i\geq 1$, $\mathcal{F}_{1,i}$ is a finite sub-family of $\left\{B_n\right\}_{n\geq 1},$\mk
\item[•] for all $i \geq 1$, for all $L\in\mathcal{F}_{1,i}$, $\vert L\vert ^{\beta+\varepsilon_i}\leq \mu(L)\leq \vert L\vert^{\alpha-\varepsilon_i},$\mk
\item[•] $\mu\left(\bigcup_{i\in\mathbb{N}}\bigcup_{L\in\mathcal{F}_{1,i}}L\right)=1.$
\end{itemize}

Note that for each $i\in\mathbb{N}$, only a finite number of balls $L \in\mathcal{F}_1 :=\bigcup_{j\in\mathbb{N}} \mathcal{F}_{1,j}$ verifies (for that $\varepsilon_i$ naturally associated with those balls) $\vert L\vert ^{\beta+\varepsilon_i}\leq \mu(L)\leq \vert L\vert^{\alpha-\varepsilon_i}$.

At step 2, a family of balls $(\mathcal{F}_{2,i})_{i\in\mathbb{N}}$ will be constructed such that:
\begin{itemize}
\item[•] for all $i,j\geq 1$, $L\in\mathcal{F}_{2,i}$, $L^{\prime}\in\mathcal{F}_{2,j}$, $L\neq L^{\prime}=\emptyset,$\medskip
\item[•] for all $i\geq 1$, $\mathcal{F}_{2,i}$ is a finite sub-family of $\left\{B_n\right\}_{n\geq 2},$\mk
\item[•] for all $i \geq 1$, for all $L\in\mathcal{F}_{2,i}$, $\vert L\vert ^{\beta+\varepsilon_{i+1}}\leq \mu(L)\leq \vert L\vert^{\alpha-\varepsilon_{i+1}},$
\item[•] one has
\begin{equation}
\label{eqfulmes}
\mu\left(\bigcup_{i\in\mathbb{N}}\bigcup_{L\in\mathcal{F}_{2,i}}L\right)=1.
\end{equation}
\end{itemize}

Write $\mathcal{F}_2 =\bigcup_{i \geq 1}\mathcal{F}_{2,i}$. Note that the family of balls $\mathcal{F}_2$ verifies, by construction, that any  $L\in\mathcal{F}_{2}$ the natural $\varepsilon_i$ associated with $L$ is never equal to $\varepsilon_1$, so that only some balls constructed in step 1 are associated with  $\varepsilon_1$ . 

The other steps are achieved following the same scheme.

The construction is detailed below:

 \mk

\textbf{Step 1:}

\mk

Let $\Omega_{1,1}=\mathbb{R}^d .$

\mk

\textbf{Sub-step 1.1:}

\mk

By Lemma \ref{gscainf} and Lemma \ref{gscaup} applied to $\varepsilon=\varepsilon_1$, there exists a $\mu$-a.c sub-sequence  $(B_{\psi_{1,1}(n)})_{n\in\mathbb{N}}$, satisfying, for every $n\in\mathbb{N}$,  
$$\vert B_{\psi_{1,1}(n)}\vert^{\beta+\varepsilon_1}\leq  \mu(B_{\psi_{1,1}(n)})\leq \vert B_{\psi_{1,1}(n)} \vert^{\alpha-\varepsilon_1}.$$ 
 By Lemma \ref{covO} applied to  $\Omega_{1,1}$, the sequence $(B_{\psi_{1,1}(n)})_{n\in\mathbb{N}}$ and $g=1$, there exists an integer $N_{1,1}$ as well as some balls $L_{1,1,1} ,...,L_{1,1,N_{1,1}}\in\left\{B_n\right\}_{n\geq 1}$ verifying:
\begin{itemize}
\item[•] for all $1\leq i<j\leq N_{1,1}$,  $L_{1,1,i} \cap L_{1,1,j}=\emptyset,$\mk
\item[•] for all $1\leq i\leq N_{1,1}$, $\vert L_{1,1,i}\vert^{\beta+\varepsilon_1}\leq \mu(L_{1,1,i})\leq \vert L_{1,1,i}\vert^{\alpha-\varepsilon_1},$\mk
\item[•] $\mu(\bigcup_{1\leq i\leq N_{1,1}}L_{1,1,i})\geq \frac{1}{2}.$\mk
\end{itemize}
Set $\mathcal{F}_{1,1}=\left\{L_{1,1,i}\right\}_{1\leq i\leq N_{1,1}}.$ 

\mk

\textbf{Sub-step 1.2:}

\mk

Let $\Omega_{1,2}=\Omega_{1,1}\setminus \bigcup_{L\in\mathcal{F}_{1,1}}L.$

By Lemma \ref{gscainf} and Lemma \ref{gscaup} with $\varepsilon=\varepsilon_2$, there exists a $\mu$-a.c sub-sequence $(B_{\psi_{1,2}(n)})_{n\in\mathbb{N}}$ satisfying 
$$\vert B_{\psi_{1,2}(n)}\vert^{\beta+\varepsilon_2}\leq \mu(B_{\psi_{1,2}(n)})\leq \vert B_{\psi_{1,2}(n)} \vert^{\alpha-\varepsilon_2}.$$ 

One applies Lemma \ref{covO} to the open set $\Omega_{1,2}$, the sub-sequence of balls $(B_{\psi_{1,2}(n)})_{n\in\mathbb{N}}$ and $g=1$. There exists $N_{1,2}\in\mathbb{N}$ such that $L_{1,2,1},...,L_{1,2,N_{1,2}}$ verifies:
  \begin{itemize}
\item[•] for all $1\leq i<j\leq N_{1,2}$,  $L_{1,2,i} \cap L_{1,2,j}=\emptyset,$\mk
\item[•] for all $1\leq i\leq N_{1,2}$, $\vert L_{1,2,i}\vert^{\beta+\varepsilon_2}\leq \mu(L_{1,2,i})\leq \vert L_{1,2,i}\vert^{\alpha-\varepsilon_2},$\mk
\item[•] $\mu(\bigcup_{1\leq i\leq N_{1,2}}L_{1,2,i})\geq \frac{1}{2}\mu(\Omega_{1,2}).$\mk
\end{itemize}
The family $\mathcal{F}_{1,2}$ is defined as
$\mathcal{F}_{1,2}=\left\{L_{1,2,i}\right\}_{1\leq i\leq N_{1,2}}.$

Proceeding iteratively as  Sub-step $1.1$ and Sub-step $1.2$, for any $i\in\mathbb{N}$, at Sub-step $1.i$ a family of balls $(\mathcal{F}_{1,i})_{i\in\mathbb{N}}$ is constructed so that it verifies:
\begin{itemize}
\item[•] for all $i,j\geq 1$, $L\in\mathcal{F}_{1,i}$, $L^{\prime}\in\mathcal{F}_{1,j}$, if $L\neq L^{\prime},$ then $L\cap L^{\prime}=\emptyset,$\medskip
\item[•] for all $i\geq 1$, $\mathcal{F}_{1,i}$ is a finite subset of $\left\{B_n\right\}_{n\geq 1},$\mk
\item[•] for all $i \geq 1$, for all $L\in\mathcal{F}_{1,i}$, $\vert L\vert ^{\beta+\varepsilon_i}\leq \mu(L)\leq \vert L\vert^{\alpha-\varepsilon_i},$
\item[•] $\mu\left(\bigcup_{i\in\mathbb{N}}\bigcup_{L\in\mathcal{F}_{1,i}}L\right)=1.$
\end{itemize}
Recall that, to justify the last item, this recursive scheme allows to cover  $\mathbb{R}^d$, up to a set of  $\mu$-measure 0 (the argument is similar to the one developed  at the end of the proof of Lemma \ref{covO} to obtain \eqref{usefulmath}).

Set $\mathcal{F}_1 =\bigcup_{i\geq 1}\mathcal{F}_{1,i}.$  With each ball $L\in\mathcal{F}$ is naturally associated a positive real number $\varepsilon(L)$, such that $\varepsilon(L)=\varepsilon_i $ if $L\in\mathcal{F}_{1,i}.$

Let us notice that the construction of the family $\mathcal{F}_2$ does not rely on the existence of the family $\mathcal{F}_1$, so that the families $\mathcal{F}_k$ can actually be built independently, following the same scheme, as described below.

\mk

\textbf{Step $k$:}

\mk

As in step 1, one constructs a family of balls  $(\mathcal{F}_{k,i})_{i\geq 1}$ verifying: 
\begin{itemize}
\item[•] for all $i,j\geq 1$, $L\in\mathcal{F}_{k,i}$, $L^{\prime}\in\mathcal{F}_{k,j}$,  $L\neq L^{\prime},$  $L\cap L^{\prime}=\emptyset,$\medskip
\item[•] for all $i\geq 1$, $\mathcal{F}_{k,i}$ is a finite subset of  $\left\{B_n\right\}_{n\geq k},$\mk
\item[•] for all $i \geq 1$, for all $L\in\mathcal{F}_{k,i}$, $\vert L\vert ^{\beta+\varepsilon_{i+k}}\leq \mu(L)\leq \vert L\vert^{\alpha-\varepsilon_{i+k}},$\mk
\item[•] one has
\begin{equation}
\label{eqfulmesbis}
\mu\left(\bigcup_{i\in\mathbb{N}}\bigcup_{L\in\mathcal{F}_{k,i}}L\right)=1.
\end{equation}
\end{itemize}
Set $\mathcal{F}_{k}=\bigcup_{i\geq 1}\mathcal{F}_{k,i}$ and 
$\mathcal{F}=\bigcup_{k\geq 1}\mathcal{F}_k.$

Denote by  $(B_{ \phi(n)})_{n\in\mathbb{N}}$ the sub-sequence of balls that constitutes the family $\mathcal{F}.$

By construction, for all $i\in\mathbb{N}$, only a finite number of balls  $L\in\mathcal{F}$ verifies $\varepsilon(L)=\varepsilon_i$ (and $\vert L\vert ^{\beta+\varepsilon_i}\leq \mu(L)\leq \vert L\vert^{\alpha -\varepsilon_i}$). In particular, for all $\varepsilon>0$, there exists $N$ large enough so that, for every $n\geq N$, $\varepsilon_n \leq \varepsilon$. Similarly, there exists $N^{\prime}\in\mathbb{N}$   so large that for every $n^{\prime}\geq N^{\prime}$,
$$\vert B_{\phi(n^{\prime})}\vert^{\beta+\varepsilon}\leq \mu(B_{ \phi(n^{\prime})})\leq \vert B_{\phi(n^{\prime})} \vert^{\alpha-\varepsilon}.$$
It follows that 
$$\alpha-\varepsilon\leq \liminf_{n\rightarrow+\infty}\frac{\log \mu(B_{\phi(n)})}{\log\vert B_{\phi(n)}\vert}\leq \limsup_{n\rightarrow+\infty}\frac{\log \mu(B_{\phi(n)})}{\log \vert B_{\phi(n)}\vert}\leq \beta +\varepsilon.$$
Letting $\varepsilon\to 0$ shows that
$$\underline{\dim}_H (\mu)\leq \liminf_{n\rightarrow+\infty}\frac{\log \mu(B_{\phi(n)})}{\log\vert B_{\phi(n)}\vert}\leq \limsup_{n\rightarrow+\infty}\frac{\log \mu(B_{\phi(n)})}{\log \vert B_{\phi(n)}\vert}\leq \overline{\dim}_H (\mu).$$

It only remains to prove that $(B_{\phi(n)})_{n\in \mathbb{N}}$ is $\mu$-a.c.

Let $\Omega $ be an open set and $g\in\mathbb{N}$. We find a finite family of balls $\left\{L\right\}_{i\in\mathcal{I}}\subset \left\{B_{\phi(n)}\right\}_{n\geq g}$ satisfying the conditions of  Definition \ref{ac}. 

 Note that, by \eqref{eqfulmes}, 
$$\mbox{setting } E=\bigcap_{k\geq 1}\bigcup_{L\in\mathcal{F}_k}L, \ \ \mbox{ then } \ \mu\left(E\right)=1.$$  
Let $x\in\Omega \cap E$ and $r_x >0$ small enough so that $B(x,r_x)\subset \Omega .$ Consider $k_x \geq \phi(g)\geq g$ large enough so that, for all $n\geq k_x$, $\vert B_n \vert \leq 2 r_x$. Recall that  $\mathcal{F}_{k_x}\subset \left\{B_n\right\}_{n\geq k_x}.$ Finally,  let us fix $k$ large enough so that $\mu(\widehat{E})\geq \frac{\mu(\Omega)}{2},$  where $\widehat{E}=\left\{x\in E \ : \ k_x \leq k\right\}$. For $x\in\widehat{E}$, let $L_x \in\mathcal{F}_k$ be the ball that contains $x$ (the balls of $\mathcal{F}_k$ being pairwise disjoint, $L_x$ is well defined) and $\left\{L_i\right\}_{i\geq 1}=\left\{L_x : x\in\widehat{E}\right\}$. One has 
\begin{itemize}
\item[•] for all $1\leq i<j$, $L_i \cap L_j =\emptyset,$ 
\item[•] for all $i\in\mathbb{N}$, $L_i \in\left\{B_{\phi(n)}\right\}_{n\geq g}$ and $L_i \subset \Omega,$ 
\item[•] $\mu(\bigcup_{i \geq 1}L_i)\geq \mu(\widehat{E})\geq \frac{\mu(\Omega)}{2}.$
\end{itemize}  
By $\sigma$-additivity, there exists $N\in\mathbb{N}$ such that $\mu(\bigcup_{1\leq i\leq N}L_i)\geq \frac{\mu(\Omega)}{4}$, which proves that $(B_{\phi(n)})_{n\in\mathbb{N}}$ satisfies  Definition \ref{ac} with $C=\frac{1}{4}$ and is indeed $\mu$-a.c. 
\end{proof}

One finishes this section with the following proposition, which supports the idea that, roughly speaking, for an $\alpha$ exact-dimensional measure $\mu$ and a $\mu$-a.c sequence of balls $(B_n)$, considering balls $(B_n)_{n\in\mathbb{N}}$ which does not verify $\mu(B_n)\approx \vert B_n \vert^{\alpha}$ is not relevant from the $\mu$-standpoint.

\begin{proposition}
Let $\mu\in\mathcal{M}( \mathbb{R}^d)$ be an $\alpha$ exact-dimensional measure  and $(B_n)_{n\in\mathbb{N}}$ a sequence of balls  satisfying $\vert B_n \vert\to 0.$ Let $\varepsilon>0.$ Let is also define $\mathcal{B}_{>}^{\varepsilon}=\left\{B_n : \mu(B_n)\leq \vert B_n \vert ^{\alpha+\varepsilon}\right\}$ and  $\mathcal{B}_{<}^{\varepsilon}=\left\{B_n : \mu(B_n)\geq \vert B_n \vert ^{\alpha-\varepsilon}\right\}.$ Then
\begin{enumerate}
\item  for any  $v<1$, $\mu(\limsup_{B\in\mathcal{B}_{>}^{\varepsilon}}vB)=0,$\mk
\item $\mu(\limsup_{B\in\mathcal{B}_{<}^{\varepsilon}}B)=0.$
\end{enumerate}
\end{proposition}

\begin{proof}
(1) Suppose that there exists $0<v<1$ such that $\mu(\limsup_{B\in\mathcal{B}_{>}^{\varepsilon}}vB)>0.$ Then, since $\mu$ is assumed to be exact-dimensional, there exists $x\in\limsup_{B\in\mathcal{B}_{>}^{\varepsilon}}vB$ such that $\lim_{r\to 0}\frac{\log \mu(B(x,r))}{\log r}=\alpha.$

Consider $r_x>0$ small enough so that, for any $0<r\leq r_x$, $\mu(B(x,r))\geq r^{\dim (\mu)+\frac{ \varepsilon}{2}}$ and $(\frac{1-v}{2})^{\dim (\mu)+\frac{\varepsilon}{2}}\geq r_x ^{\frac{\varepsilon}{4}}.$ Let also $n$ be  large enough so that $x\in B_n$ and $\vert B_n \vert \leq r_x .$ Then $  B(x,\frac{(1-v)}{2} \vert B_n \vert)\subset B_n$, so that
\begin{align}
\label{cone}
\mu(B_n)\geq \mu(B(x,\frac{1-v}{2}\vert B_n \vert))\geq \vert B_n \vert^{\alpha+\frac{\varepsilon}{2}}(\frac{1-v}{2})^{\alpha+\frac{\varepsilon}{2}}\geq \vert B_n \vert^{\alpha+\frac{3\varepsilon}{4}}.
\end{align}
This contradicts the definition of $\mathcal{B}_{>}^{\varepsilon}.$

\medskip
\noindent(2) Assume that $\mu(\limsup_{B\in\mathcal{B}_{<}^{\varepsilon}}B)>0.$ Then, again, there exists $x\in\limsup_{B\in\mathcal{B}_{<}^{\varepsilon}}B$ so  that $\lim_{r\to 0}\frac{\log \mu(B(x,r))}{\log r}=\alpha.$ Consider $r_x >0$ small enough so such that, for any $0<r\leq r_x$, $\mu(B(x,r))\leq r^{\alpha-\frac{ \varepsilon}{2}}.$ Consider $n\in\mathbb{N}$ large enough so that $x\in B_n$ and $\vert B_n \vert \leq r_x.$ One has $B_n \subset B(x,\vert B_n \vert)$, hence 
\begin{align*}
\mu(B_n)\leq \mu(B(x,\vert B_n \vert))\leq \vert B_n \vert^{\alpha-\frac{\varepsilon}{2}}.
\end{align*}
This contradicts the definition of $\mathcal{B}_{<}^{\varepsilon}.$ 
\end{proof}
\begin{remarque}
For doubling measures, it is straightforward that item $(1)$ can be replaced by simply $\mu(\limsup_{B\in\mathcal{B}_{>}^{\varepsilon}}B)=0.$ It can be proved that this is also the case for 1-average d-1 unrectifiable measures (as a consequence of \cite[Theorem 2.11]{KS}). Some  self-similar measures with open set condition satisfies this property (see \cite{KS} again for more details).
\end{remarque}


\

\section{Some explicit examples}
\label{sec-exam}
In this section, applications of Theorem \ref{theoremextra} are given.

\subsection{Rational approximation}
Let us recall the following result from Hurwitz see \cite{HW}, p 219, for more details.

\begin{theoreme}
\label{hure}
Let $x\in[0,1]\setminus \mathbb{Q}.$ There exists an infinite number of pairs $(p,q)\in\mathbb{N}\times \mathbb{N}^*$ with $p\wedge q=1$ and 
\begin{equation}
\label{hur}
\Big\vert x -\frac{p}{q}\Big\vert<\frac{1}{\sqrt{5}q^2}.
\end{equation}
\end{theoreme}

An immediate corollary of  Theorem \ref{equiac},  Theorem \ref{theoremextra} and Theorem \ref{hure} is the following:


\begin{corollary}
Let $\mu \in\mathcal{M}([0,1])$ be any diffuse measure. Then the sequence of balls $\Big(B(\frac{p}{q},\frac{1}{q^2})\Big)_{0\leq p\leq q ,  q\in\mathbb{N}^* ,p\wedge q =1}$ is weakly redundant (see \cite{BS2}) and $\mu$-a.c. In particular, if $\mu$ is $\alpha$-exact-dimensional, for $0\leq \alpha\leq d$,  then there exists a sequence $(\varepsilon_n )_{n\in\mathbb{N}}\in(\mathbb{R}^* _+ )^{\mathbb{N}} $ with $ \lim_{n\rightarrow+\infty}\varepsilon_n =0$ and an infinite number of pairs $(p_n ,q_n)\in\Big(\mathbb{N}\times\left\{0,...,q_n\right\}\Big)^{\mathbb{N}} $ such that $p_n \wedge q_n =1$ and $$\Big(\frac{1}{q_n^{2}}\Big)^{\alpha+\varepsilon_n}\leq\mu\Big(B\Big(\frac{p_n}{q_n},\frac{1}{q_n ^2}\Big)\Big)\leq \Big(\frac{1}{q_n^{2}}\Big)^{\alpha-\varepsilon_n}.$$
Moreover, writing $$\mathcal{B}_{\mu}=\left\{ B(\frac{p_n}{q_n},\frac{1}{q_n ^2}):\Big(\frac{1}{q_n^{2}}\Big)^{\alpha+\varepsilon_n}\leq\mu\Big(B\Big(\frac{p_n}{q_n},\frac{1}{q_n ^2}\Big)\Big)\leq \Big(\frac{1}{q_n^{2}}\Big)^{\alpha-\varepsilon_n}\right\}_{n\in\mathbb{N}},$$ one has
$$\dim_H(\limsup_{B\in\mathcal{B}_{\mu}}B)=\alpha.$$
\end{corollary}

\subsection{Application to Random balls}


Let us recall Shepp's Theorem of Shepp \cite{shepp1972randomcovering}.


\begin{theoreme}
\label{she}
Let $(l_n)_{n\in\mathbb{N}}\in(\mathbb{R}_+ ^{*} )^{\mathbb{N}}$ and $(X_n )_{n\in\mathbb{N}}$ be a sequence of i.i.d uniformly distributed random variables on $[0,1]$. Then
$$
 \limsup_{n\rightarrow+\infty}B(X_n, l_n)=[0,1] \Leftrightarrow \sum_{n\geq 0}\frac{1}{n^2}\exp(l_1 +...+l_n)=+\infty. $$    
\end{theoreme}


From Theorem \ref{equiac}, Theorem \ref{theoremextra} and Theorem \ref{she}, the following corollary is deduced.


\begin{corollary}
For any $\alpha$ exact-dimensional measure $\mu \in\mathcal{M}([0,1])$, $0\leq\alpha\leq d$, for almost any i.i.d sequence of random variables uniformly distributed on $[0,1]$, $(X_n)_{n\in\mathbb{N}}$, there exists a sequence of positive real numbers $(\varepsilon_k )_{k\in\mathbb{N}}$ with $\varepsilon_k \to 0$ and a subsequence $(n_k)_{k\in\mathbb{N}}\to+\infty$ satisfying
$$\Big(\frac{2}{n_k} \Big)^{\alpha+\varepsilon_k}\leq \mu\Big(B\Big(X_{n_k} ,\frac{2}{n_k} \Big)\Big)\leq \Big(\frac{2}{n_k} \Big)^{\alpha-\varepsilon_k}.$$ 
Writing again $$\mathcal{B}_{\mu}=\left\{B\Big(X_{n_k} , \frac{2}{n_k}\Big):\Big(\frac{2}{n_k}\Big)^{\alpha+\varepsilon_k}\leq \mu\Big(B\Big(X_{n_k} ,\frac{2}{n_k} \Big)\Big)\leq \Big(\frac{2}{n_k} \Big)^{\alpha-\varepsilon_k}\right\}_{k\in\mathbb{N}},$$
one has $$\dim_H (\limsup_{B\in\mathcal{B}_{\mu}}B)=\alpha.$$ 
\end{corollary}


\subsection{Examples in dynamical systems}


Let us introduce some notation.

Let $m\geq 2$  and $S=\left\{f_1 ,...,f_m\right\}$ be a system of $m$ similarities of $\mathbb{R}^d \to \mathbb{R}^d$ of ratio of contraction $0<c_1 <1 ,...,0<c_m <1.$

Let us also write $\Lambda=\left\{1,...,n\right\}$, $\Lambda^{*}=\bigcup_{k\geq 0}\Lambda^k$ and for $\underline{i}=(i_1,...,i_k)\in\Lambda^k$,
\begin{itemize}
\item[•]$f_{\underline{i}}=f_{i_1}\circ ... \circ f_{i_k},$\mk
\item[•] $X_{\underline{i}}=f_{\underline{i}}([0,1]^d),$\mk
\item[•]$c_{\underline{i}}=c_{i_1}\times ... \times c_{i_k}.$
\end{itemize}

Let us fix also $(p_1,...,p_m)\in\mathbb{R}_+ ^{\mathbb{N}}$  a probability vector, i.e a vector  verifying $\sum_{1\leq i\leq m}p_i =1$.




\begin{remarque}
\label{covks}
Let $\mu$ defined by \eqref{equass} and $x\in K_S$. Then, for any $k\in\mathbb{N}$, the balls $\left\{B(f_{\underline{i}}(x),2\vert K_S \vert c_{\underline{i}})\right\}_{\underline{i}\in\Lambda^k}$ covers $K_S$. In particular, by Theorem \ref{equiac} $\Big(B(f_{\underline{i}}(x),3\vert K_S \vert c_{\underline{i}})\Big)_{\underline{i}\in\Lambda^*}$ is $\mu$-a.c.
\end{remarque}



As a consequence of Theorem \ref{theoremextra} and Lemma \ref{covks}, one gets:


\begin{corollary}
Let $\mu\in\mathcal{M}(\mathbb{R}^d)$ be a measure defined by \eqref{equass} and $x\in K_S$. There exists a $\mu$-a.c  weakly redundant sub-sequence of balls $(B_{n})_{n\in\mathbb{N}}$ extracted from $\Big(B(f_{\underline{i}}(x),3c_{\underline{i}})\Big)_{\underline{i}\in\Lambda^*}$ such that, for all $n\in\mathbb{N}$ and for some sequence $(\varepsilon_k)_{k\in\mathbb{N}}\in (\mathbb{R}^{*}_+ )^{\mathbb{N}}$ verifying $\varepsilon_k \to 0$, 
$$\vert B_n \vert^{\dim(\mu)+\varepsilon_n}\leq \mu(B_n)\leq \vert B_n \vert^{\dim(\mu)-\varepsilon_n},$$
and $\mathcal{B}_{\mu}=\left\{B_n \right\}_{n\in\mathbb{N}}$ satisfies
$$ \dim_H (\limsup_{n\rightarrow+\infty}B_n)=\dim (\mu).$$

\end{corollary}
 \

 \section{Proofs of Theorem \ref{majoss}}
\label{sec-upper}

The proof strongly relies on the following result proved in \cite{ED3}.
\begin{theoreme}
\label{contss}
Let $S$ be a self-similar IFS of $\R^d$. Let $K$ be the attractor of $S$. Let~$\mu$ be a self-similar measure  associated with $S$. For any $0\leq s<\dim(\mu)$, there exists a constant $c=c(d,\mu,s)>0 $ depending on the dimension $d$, $\mu$ and  $s$  only, such that for any ball $B=B(x,r)$ centered on $K$ and $r\leq 1$, any open set $\Omega$, one has 
\begin{align}
\label{genhcontss}
&c(d,\mu,s)\vert B\vert ^{s}\leq\mathcal{H}^{\mu, s }_{\infty}(\widering{B})\leq  \mathcal{H}^{ \mu, s}_{\infty}(B)\leq\vert B\vert ^{s}\text{ and } \nonumber\\
&c(d,\mu,s)\mathcal{H}^{s}_{\infty}(\Omega \cap K)\leq \mathcal{H}^{\mu,s}_{\infty}(\Omega)\leq \mathcal{H}^{s}_{\infty}(\Omega \cap K).
\end{align}
For any $s>\dim(\mu)$, $ \mathcal{H}^{\mu,s}_{\infty}(\Omega)=0.$

\end{theoreme}

Recall that the sequence $\mathcal{B}=(B_n)_{n\in\mathbb{N}}$ is assumed to be weakly redundant. In such a case, for any $\varepsilon>0$, following the notation involved in Definition \ref{wr}, it holds that
\begin{align*}
\sum_{n\geq 0}\vert B_n \vert^{\varepsilon}\mu(B_n)=\sum_{k\geq 0}\sum_{B\in\mathcal{T}_k (\mathcal{B})}\vert B \vert^{\varepsilon}\mu(B)\leq \sum_{k\geq 0}\sum_{1\leq j\leq J_k}2^{-k\varepsilon}\sum_{B\in T_{k,j} (\mathcal{B})} \mu(B).
\end{align*}
Since for every $(k,j)$ the family $T_{k,j}$ is composed of pairwise disjoint balls, this yields
\begin{equation}
\label{majowr}
\sum_{n\geq 0}\vert B_n \vert^{\varepsilon}\mu(B_n)\leq  \sum_{k\geq 0}\sum_{1\leq j\leq J_k}2^{-k\varepsilon}=\sum_{k\geq 0}J_k 2^{-k\varepsilon}<+\infty.
\end{equation}

Now, for $n \in\mathbb{N}$, consider a sequence of balls $(A_k ^{n})_{k\in\mathbb{N}}$, with $\vert A_k ^n \vert \leq \vert B_n \vert$ and such that $U_n \cap K\subset \bigcup_{k\geq 0}A_k ^n$. Recall Theorem \ref{contss} and its notations. One  has.
\begin{equation}
\label{equapmaj1}
\mathcal{H}^s_{\infty}(U_n)\leq \sum_{k\geq 0}\vert A_{k}^n  \vert^s \leq 2\mathcal{H}^s_{\infty}(U_n)\leq \frac{2}{c(d,\mu,s)}\mathcal{H}^{\mu,s}_{\infty}(U_n)\leq  \frac{2}{c(d,\mu,s)}\mu(B_n).
\end{equation}

Since for each $n\in\mathbb{N}$, $U_n \cap K \subset \bigcup_{k\geq 0}A_k ^n$, it holds that $\limsup_{n\rightarrow+\infty}U_n \cap K\subset \limsup_{k,n \rightarrow+\infty}A_k ^n .$ For any $\varepsilon>0$, one gets
\begin{align*}
\sum_{n\geq 0}\sum_{k\geq 0}\vert A_k ^n\vert^{s+\varepsilon}\leq \sum_{n\geq 0}\vert B_n \vert^{\varepsilon}\frac{2}{c(d,\mu,s)}\mu(B_n).
\end{align*}
In particular, by \eqref{majowr}, 
\begin{equation}
\sum_{n\geq 0}\sum_{k\geq 0}\vert A_k ^n\vert^{s+\varepsilon}<+\infty.
\end{equation}

One concludes that

$$\mathcal{H}^{s+\varepsilon}(\limsup_{n\rightarrow+\infty}U_n \cap K)\leq  \mathcal{H}^{s+\varepsilon}(\limsup_{k,n\rightarrow+\infty}A_k^n)<+\infty .$$
This implies that $\dim_H (\limsup_{n\rightarrow+\infty}U_n \cap K)\leq s+\varepsilon$ and $\varepsilon$ being arbitrary,  $$\dim_H (\limsup_{n\rightarrow+\infty}U_n \cap K)\leq s .$$

%
%
%

\begin{remarque}
\label{remajo1}
\item[•] An important fact to underline here is that the convergences  established in \eqref{majowr} and  \eqref{equapmaj1} do not rely on the fact that the measure is self-similar, but hold for any measure $\mu$. One could state a comparable upper-bound Theorem for any measure $\mu$ by replacing $K$ by a $G_{\delta}$ set of full measure in \eqref{uppernound18}. 

These computations also have the following straightforward consequence for a measure $\mu\in\mathcal{M}(\mathbb{R}^d)$ without the self-similarity assumption: Assume that, for $n$ large enough, $\mathcal{H}^{\mu,s}_{\infty}(U_n)\leq \mu(B_n).$ If the sequence $(U_n)_{n\in\mathbb{N}}$  verifies that for any ball $B_i \subset U_n$ one also has $B_i \subset \bigcup_{k\geq n}A_k^n$ (where the balls $(A_{k,n})_{k\in\mathbb{N}}$ are chosen as in the proof of Theorem \ref{majoss}), then $\dim_H (\limsup_{n\rightarrow+\infty}U_n)\leq s$. In particular if this holds for any $s>s(\mu,\mathcal{B},\mathcal{U}),$ then $\dim_H (\limsup_{n\rightarrow+\infty}U_n)=s(\mu,\mathcal{B},\mathcal{U}).$\medskip

\item[•] When the self-similar measure verifies $\supp (\mu)=[0,1]^d$, the existence of $s_0$ as in Corollary \ref{equass1} is ensured as soon as the shapes of the sets $U_n$ are ``uniform'' in $n$. For instance, consider the case where $\mu=\mathcal{L}^d$ and  $(U_n =R_n )_{n\in\mathbb{N}}$, where $R_n$ is an open rectangle associated with some vector $\boldsymbol{\tau}=(\tau_1 ,...,\tau_d)$ defined as in Theorem \ref{rectssmajo}. Recall that by Theorem \ref{contss}, the Lebesgue essential Hausdorff content and the classical Hausdorff content are equivalent. It is easily verified that, for any $n\in\mathbb{N}$, $\mathcal{H}^{s}_{\infty}(R_n)=\vert B_n \vert^{g_{\boldsymbol{\tau}}(s)}$, for the  mapping $g_{\boldsymbol{\tau}}:\mathbb{R}^+ \to \mathbb{R}^+$, defined as (see \cite{KR})
$$ g_{\boldsymbol{\tau}}(s)= \max_{1\leq k\leq d}\left\{s\tau_k -\sum_{1\leq i\leq k}\tau_k-\tau_i\right\}.$$ 
Note that $g_{\boldsymbol{\tau}}(s) $  does not depend on $n$. Corollary can therefore be applied with $s_0=\min_{s: g_{\boldsymbol{\tau}}(s)\geq d}\left\{s\right\}$. \mk
\item[•] Unfortunately, when such an $s_0$ does not exist, the Hausdorff dimension of $\limsup_{n\rightarrow+\infty}U_n$ has to depend on the structure of the sequence $(U_n)$ itself. 

Consider  $0<s_1 < s_2 \leq d$ and two vectors $\boldsymbol{\tau_1}$ and $\boldsymbol{\tau_2}$ such that $s_1=\min_{s: g_{\boldsymbol{\tau_1}}(s)\geq d}\left\{s\right\}$ and $s_2=\min_{s: g_{\boldsymbol{\tau_2}}(s)\geq d}\left\{s\right\}.$ Consider a weakly redundant sequence of balls $(B_n)_{n\in\mathbb
N}$ of $[0,1]^d$ and a sequence of open sets $(U_n)_{n\in\mathbb{N}}$, $U_n \subset B_n$ satisfying:\mk
\begin{itemize}
\item[•] $\vert B_n \vert \to 0,$\mk
\item[•] $\mathcal{L}^d (\limsup_{n\rightarrow+\infty}B_n)=1,$\mk
\item[•] for any $n\in\mathbb{N}$, $B_n \subset [0,\frac{1}{2})\times \prod_{i=2}^d [0,1]$ or $B_n \subset (\frac{1}{2},1]\times \prod_{i=2}^d [0,1],$\mk
\item[•] for any $n\in\mathbb{N}$ such that  $B_n \subset [0,\frac{1}{2})\times \prod_{i=2}^d [0,1]$, $U_n =(R_n )$ with $R_n$ an open rectangle associated with $\boldsymbol{\tau_1}$ as in Theorem \ref{rectssmajo},\mk
\item[•] for any $n\in\mathbb{N}$ such that $B_n \subset (\frac{1}{2},1]\times \prod_{i=2}^d [0,1],$  $U_n =R_n $ with $R_n$ an open rectangle associated with $\boldsymbol{\tau_2}.$\mk
\end{itemize} 
Then smallest real number such that the condition of Theorem \ref{majoss} holds is $s_2$, the largest real number such that the condition of Theorem \ref{zzani} holds is $s_1$ and $\dim_H (\limsup_{n\rightarrow+\infty}U_n)=s_2.$

On the other hand, following the scheme of example 3.5 in \cite{KR}, it is also possible to construct two weakly redundant sequences of balls $(B_{n,1})_{n\in \mathbb{N}}$ and $(B_{n,2})_{n\in\mathbb{N}}$ such that:
\begin{itemize}
\item[•] $\vert B_{n,1} \vert \to 0$ and $\vert B_{n,2}\vert \to 0,$\mk
\item[•] $0<\mathcal{L}^d (\limsup_{n\rightarrow+\infty}B_{n,2})<1,$\mk
\item[•]  $\limsup_{n\rightarrow+\infty}R_{n,2}=\emptyset$, where $R_{n,2}\subset B_{n,2}$ is an open rectangle associated with $\boldsymbol{\tau_2}$,\mk
\item[•] $\limsup_{n\rightarrow+\infty}B_{n,1}\subset [0,1]^d \setminus \limsup_{n\rightarrow+\infty}B_{n,2}$ and $$\mathcal{L}^{d}(\limsup_{n\rightarrow+\infty}B_{n,1})=1-\mathcal{L}^d (\limsup_{n\rightarrow+\infty}B_{n,2}).$$ For any $n\in\mathbb{N}$, denote by $R_{n,1}\subset B_{n,1}$  an open rectangle associated with $\boldsymbol{\tau_1}.$ 
\end{itemize} 
Those properties implies that that any sequence $(B_n)_{n\rightarrow+\infty}$ corresponding to the family $\left\{B_{n,i}\right\}_{n\in\mathbb{N}, i\in\left\{1,2\right\}}$ is weakly redundant and satisfies $\mathcal{L}^d (\limsup_{n\rightarrow+\infty}B_n)=1.$

  Again, the smallest real number such that the condition of Theorem \ref{rectssmajo} holds is $s_2$, the largest real number such that the condition of Theorem \ref{zzani} holds is $s_1$ but this time, $\dim_H (\limsup_{n\rightarrow+\infty}U_n)=s_1.$

\end{remarque}

\section{Conclusion and perspectives}

\label{conclu}
 The properties stated in Theorem \ref{theoremextra} (the prescribed measure and the weak redundancy) are of course non exhaustive and maybe more can be imposed to well chosen subsequences of $\mu$-.a.c sequences of balls.  It turns out that in the quasi-Bernoulli case it is enough to get condition under which the lower-bound found in \cite{ED1} is also an upper-bound , but it is likely that in some other cases, one needs to ask the sequence to verify more properties to ensure the sharpness of a certain lower-bound. In particular, it can be proved that, under very weak hypothesis on a $\mu$-a.c sequence $(B_n)_{n\in\mathbb{N}}$, given a $G_{\delta}$ set of full measure $G$, it is possible to assume (up to an extraction) that the set $G^{\prime}=\limsup_{n\rightarrow+\infty}B_n$ is a $G_{\delta}$ set of full measure with $G^{\prime}\subset G$ (so that one can assume that $\limsup_{n\rightarrow+\infty}B_n$ is always included in any $G_{\delta}$ of full measure if needed). 

\mk

\bibliographystyle{plain}
\bibliography{bibliogenubi}

\begin{thebibliography}{10}

\bibitem{BS3}
J.~Barral and S.~Seuret.
\newblock Sums of dirac masses and conditioned ubiquity.
\newblock {\em C. R. Acad. Sci. Paris}, Sér. I 339:787--792, 2004.

\bibitem{BS}
J.~Barral and S.~Seuret.
\newblock Heterogeneous ubiquitous systems in $\mathbb{R}^d$ and {H}ausdorff
  dimensions.
\newblock {\em Bull. Brazilian Math. Soc}, 38(3):467--515, 2007.

\bibitem{BS4}
J.~Barral and S.~Seuret.
\newblock The multifractal nature of heterogeneous sums of dirac masses.
\newblock {\em Math. Proc. Cambridge Philos. Soc.}, 144(3):707--727, 2008.

\bibitem{BS2}
J.~Barral and S.~Seuret.
\newblock Ubiquity and large intersections properties under digit frequencies
  constraints.
\newblock {\em Math. Proc. Cambridge Philos. Soc.}, 145(3):527--548, 2008.

\bibitem{BVBC}
V.~Beresnevich and S.~Velani.
\newblock The divergence borel-cantelli lemma revisited.
\newblock {\em arXiv:2103.12200}, 2021.

\bibitem{BV}
V.~Beresnevitch and S.~Velani.
\newblock A mass transference principle and the {D}uffin-{S}chaeffer conjecture
  for {H}ausdorff measures.
\newblock {\em Ann. Math.}, 164(3):22 pages, 2006.

\bibitem{Be}
A.S. Besicovitch.
\newblock A general form of the covering principle and relative differentiation
  of additive functions.
\newblock {\em Proc. Cambridge Philos. Soc.}, 41:103–110, 1945.

\bibitem{ED1}
E.~Daviaud.
\newblock An anisotropic inhomogeneous ubiquity theorem.
\newblock {\em preprint}.

\bibitem{ED3}
E.~Daviaud.
\newblock An heterogeneous ubiquity theorem, application to self-similar
  measures with overlaps.
\newblock {\em preprint}, 2022.

\bibitem{F}
K.~Falconer.
\newblock {\em Fractal geometry}.
\newblock John Wiley \& Sons, Inc., Hoboken, NJ, second edition, 2003.
\newblock Mathematical foundations and applications.

\bibitem{Fed}
H.~Federer.
\newblock {\em Geometric measure theory}, volume Band 153 of {\em Die
  Grundlehren der mathematischen Wissenschaften)}.
\newblock Springer-Verlag New York Inc., New York, 1969.

\bibitem{FH}
D.~Feng and H.~Hu.
\newblock Dimension theory of iterated function systems.
\newblock {\em Comm. Pure Appl. Math.}, 62:1435--1500, 2009.

\bibitem{HV}
R.~Hill and S.~Velani.
\newblock The ergodic theory of shrinking targets.
\newblock {\em Inv. Math.}, 119:175--198, 1995.

\bibitem{HW}
R.~Holley and E.C. Waymire.
\newblock Multifractal dimensions and scaling exponents for strongly bounded
  random fractals.
\newblock {\em Ann. Appl. Probab.}, 2:819--845, 1992.

\bibitem{Hutchinson}
J.E. Hutchinson.
\newblock Fractals and self similarity.
\newblock {\em Indiana Univ. Math. J.}, 30:713--747, 1981.

\bibitem{Ja}
S.~Jaffard.
\newblock Wavelet techniques in multifractal analysis.
\newblock In {\em Fractal Geometry and Applications: A Jubilee of Beno\^{\i}t
  Mandelbrot, M. Lapidus and M. van Frankenhuijsen, Eds., Proc. Symposia in
  Pure Mathematics}, volume 72(2), pages 91--152. AMS, 2004.

\bibitem{KR}
H.~Koivusalo and M.~Rams.
\newblock Mass transference principle: From balls to arbitrary shapes.
\newblock {\em To appear in I.R.M.N}, 2020.

\bibitem{KS}
A.~Käenmäki, T.~Sahlsten, and P.~Shmerkin.
\newblock Dynamics of the scenery flow and geometry of measures.
\newblock {\em Proc. Lond. Math. Soc.}, 110(3):1248--1280, 2015.

\bibitem{LS}
L. and S.~Seuret.
\newblock Diophantine approximation by orbits of expanding markov maps.
\newblock {\em Ergod. Th. Dyn. Syst.}, 33:585--608, 2013.

\bibitem{Ma}
P.~Mattila.
\newblock {\em Geometry of Sets and Measures in Euclidean Spaces: Fractals and
  Rectifiability}.
\newblock Cambridge Studies in Advanced Mathematics, 1999.

\bibitem{PR}
T.~Persson and M.~Rams.
\newblock On shrinking targets for piecewise expanding interval maps.
\newblock {\em Ergod. Th. Dyn. Syst.}, 37:646–663, 2017.

\bibitem{shepp1972randomcovering}
L.~A. Shepp.
\newblock Covering the line with random intervals.
\newblock {\em Z. Wahrscheinlichkeitstheorie und Verw. Gebiete}, 23:163--170,
  1972.

\bibitem{rect}
B.~Wang and J.~Wu.
\newblock Mass transference principle from rectangles to rectangles in
  diophantine approximation.
\newblock {\em Mathematishe Annalen}, 110(381):1--75, 2021.

\end{thebibliography}
\end{document}